\documentclass[12pt]{amsart}

\usepackage{amsmath,graphicx,verbatim, amsfonts,amssymb,amsthm,pictexwd,comment, amscd}
 \usepackage{url}
 \usepackage{xcolor}
\usepackage{soul}
\usepackage{subfig, enumerate}
\usepackage{graphics}
\usepackage[section]{placeins}
\usepackage{enumitem}
\usepackage{hyperref}
\hypersetup{
    colorlinks=false, 
    linktoc=all,     
    linkcolor=blue,  
    linktocpage,
}
\usepackage{xcolor}
\usepackage{soul}

\newtheorem{thm}{Theorem}[section]

\newtheorem{lem}[thm]{Lemma}
\newtheorem{conj}[thm]{Conjecture}

\newtheorem*{thmu}{Theorem}

\newtheorem{defn}[thm]{Definition}

\newtheorem{lemma}[thm]{Lemma}

\theoremstyle{remark}

\makeatletter
\let\c@equation\c@thm
\makeatother
\numberwithin{equation}{section}

\newcommand{\Q}{\mathbb{Q}}
\newcommand{\R}{\mathbb{R}}
\newcommand{\Z}{\mathbb{Z}}

\newcommand{\N}{\mathbb{N}}
\newcommand{\T}{\mathbb{T}}

\renewcommand{\S}{{\mathcal S}}
\renewcommand{\i}{\iota}

\DeclareMathOperator{\Tr}{Tr}

\newcommand{\td}[1]{{\red{\textbf {[#1]}}}}
\newcommand{\red}[1]{{\color{red} #1}}

\newcommand{\old}[1]{}

\title{A Family of Minimal and Renormalizable Rectangle Exchange Maps}
\author{Ian Alevy \textsuperscript{1}}
\thanks{1. Division of Applied Mathematics, Brown University, Providence, RI 02912, USA; ian\_alevy@brown.edu.}
\author{Richard Kenyon \textsuperscript{2}}
\thanks{2. Department of Mathematics, Brown University, Providence, RI 02912, USA; rkenyon@math.brown.edu.}
\author{Ren Yi \textsuperscript{3}}
\thanks{3. Department of Mathematics, Brown University, Providence, RI 02912, USA; renyi@math.brown.edu.}

\date{\today}

\begin{document}

\begin{abstract}
A \emph{domain exchange map} (DEM) is a dynamical system defined on a smooth Jordan domain which is a piecewise translation. We explain how to use cut-and-project sets to construct minimal DEMs. Specializing to the case in which the domain is a square and the cut-and-project set is associated to a Galois lattice, we construct an infinite family of DEMs in which each map is associated to a PV number. We develop a renormalization scheme for these DEMs. Certain DEMs in the family can be composed to create multistage, renormalizable DEMs. 
\end{abstract}

\maketitle
\tableofcontents   

\section{Introduction}

A \textbf{\emph{smooth Jordan domain}} \(X\) is non-empty closed bounded set in \(\R^2\) whose boundary is a piecewise smooth Jordan curve. We construct a dynamical system on \(X\) which is a piecewise translation known as a \textbf{\emph{domain exchange map (DEM)}}. The dynamical system is a 2-dimensional generalization of an interval exchange transformation.
\begin{defn}
Let \(X\) be a Jordan domain partitioned into smaller Jordan domains, with disjoint interiors, in two different ways 
\[ X = \bigcup_{k=0}^N A_k = \bigcup_{k=0}^N B_k\]
such that for each \(k\), \(A_k\) and \(B_k\) are translation equivalent, i.e., there exists \(v_k\in \R^2\) such that \(A_k= B_k+v_k\). A \textbf{\emph{domain exchange map}} is the piecewise translation on \(X\) defined for \(x\in \mathring A_k\) by 
\[ T(x) = x+v_k.\]
The map is not defined for points \(x\in  \bigcup_{k=0}^N \partial A_k\).
\end{defn} 

In section \(2\) we explain how to use \textbf{\emph{cut-and-project sets}} to define a DEM on any smooth Jordan domain \(X\). 
\begin{defn}
Let \(L\) be a full-rank lattice in \(\R^3\) and \(X\) a domain in the \(xy\)-plane in \(\R^3\). Define 
\[ 
P = \{ \pi_z(p) ~:~ p \in L \text{ and } \pi_{xy}(p)\in X\}.
\]
where $\pi_z$ is the projection onto the $z$ axis and $\pi_{xy}$ is the projection onto the $xy$-plane.
The point set \(P\) is a \textbf{\emph{cut-and-project set}} if the following two properties are satisfied:
\begin{enumerate}
\item \(\pi_z |_L \) is injective 
\item \(\pi_{xy}(L)\) is dense in \(\R^2\).
\end{enumerate}
\end{defn}

In this setting we define $\Lambda(X, L)$ to be the set of lattice points 
\[
\Lambda(X,L)= \{ \textbf x \in L ~:~  \pi_{xy}(\textbf x)\in X\}.
\]
The projection $\pi_{xy}(\Lambda(X, L))$ is dense in $X$. 

The DEM is defined by projecting a dynamical system on $\Lambda(X, L)$ onto \(X\).
\begin{figure}[h]
  \begin{center}
    \includegraphics[width=\textwidth]{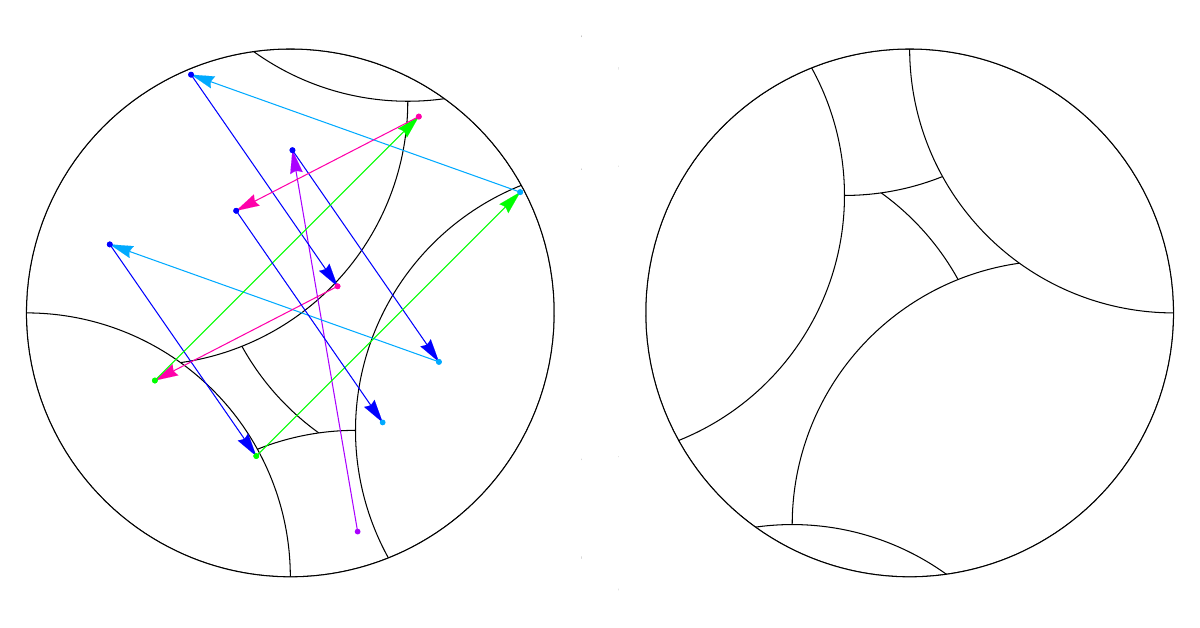}
\caption{Domain exchange map on a disk and the forward orbit of a point. One iteration of the map consists in translating each of the seven regions delineated by the black boundaries in the first panel
to its position shown in the second panel.}
\label{fig:dem-circ}
  \end{center}
\end{figure}
Figure \ref{fig:dem-circ} shows a DEM, in which \(X\) is the unit disk, constructed in this manner. The boundary of each tile is an arc of a circle with unit radius.
For almost every point $x$
the forwards and backwards orbits of \(x\) under the DEM are well-defined. We characterize the orbits of DEMs constructed using cut-and-project sets:
\begin{thm}\label{minimality}
For a DEM in $\R^2$ associated to a cut-and-project set from $\R^3$, every well-defined orbit is dense and equidistributed.
\end{thm}

The DEMs produced by our construction are amenable to analysis when the lattice and domain have a special algebraic structure. A \textbf{\emph{Pisot-Vijayaraghavan number}}, more simply called a {\bf\emph{PV number}}, is a real algebraic integer with modulus larger than \(1\) whose Galois conjugates have modulus strictly less than one. 

Let $\lambda=\lambda_3$ be a PV number whose Galois conjugates 
$\lambda_1,\lambda_2$ are real. Then 
$\Q[\lambda]$ has three embeddings into $\R$, and we can identify $\R^3$ with the product of these three embeddings, with the $x$-, $y$- and $z$-coordinates corresponding to embeddings sending $\lambda$ to $\lambda_1,\lambda_2, \lambda_3$ respectively.  Then $\Z[\lambda]$ is a lattice in $\R^3$ of the above type, and 
$$\pi_{xy}(a+b\lambda+c\lambda^2) =(a+b\lambda_1+c\lambda_1^2, a+b\lambda_2+c\lambda_2^2).$$
Multiplication by $\lambda$ is an integer transformation of $\Z[\lambda]$. We call this the \textbf{\emph{Galois embedding}} of the lattice \(\Z[\lambda]\). Note that \(\Z[\lambda]\) can be identified with \(\Z^3\) under the map 
\[ (a,b,c) \mapsto a+b \lambda + c\lambda^2.\]

When \(X\) is a smooth Jordan domain and \(L\) is the Galois embedding of a PV number whose Galois conjugates are real then the point set \(\Lambda(X,L)\) satisfies the conditions of being a cut-and-project set. We call a DEM associated to a Galois lattice a \textbf{\emph{PV DEM}}. We give a detailed analysis of PV DEMs in the case when the lattice is a Galois lattice and \(X\) is the unit square $[0,1]^2$. Since the tiles inherit their shape from the boundary of \(X\), under these assumptions the tiles are rectilinear polygons. We call these DEMs \textbf{\emph{rectangle exchange maps (REMs)}}. 

One way to construct a PV DEM is to find a \textbf{\emph{Pisot matrix}} whose eigenvalues are all real. A \textbf{\emph{Pisot matrix}} is an integer matrix with one eigenvalue greater than $1$ in modulus and the remaining eigenvalues strictly less than $1$ in modulus (in particular, its leading eigenvalue is a PV number). Define $\S$ to be the following set of matrices:
\[
\S = \left \{ M_n=\begin{bmatrix}
0 & 1 & 0 \\
0 & 0 & 1 \\
1 & -n & n+1
\end{bmatrix} : n \geq 6 \right\}
\]
We will show in Section \ref{sec:monoid} that every matrix $M_n \in \S$ is a Pisot matrix. For $M_n \in \S$, let $\lambda$ be the leading eigenvalue of $M_n$. The Galois embedding of $\Z[\lambda]$ gives rise to a PV REM (Section \ref{sec:cons}). Let \(T_{M}\) denote the PV REM associated to the Galois embedding of the eigenvalues of \(M\).

We extend the family $\{T_{M_n}: \ M_n \in \S\}$ of PV REMs to a larger family of REMs via the monoid of matrices \(\mathcal M\) consisting of nonempty products of matrices in \(\S\). Lemma \ref{lem:monoid} establishes that \(\mathcal M\) is in fact a monoid of Pisot matrices.
\begin{lem}\label{lem:monoid}
If \(W\in \mathcal M\) then its eigenvalues \(\lambda_1,\lambda_2\) and \(\lambda_3\) are real and satisfy the inequalities 
\[0< \lambda_1< \lambda_2<1<\lambda_3.\]
\end{lem}

Avila and Delecroix in \cite{ad2015} give a neat criterion for checking whether a family of matrices generates a monoid of Pisot matrices.
Even though our (computational) proof of Lemma \ref{lem:monoid} is somewhat along the same lines, 
we were not able to apply their results directly to this family.

\emph{\textbf{Admissible REMs}} are defined by a subset of \textbf{admissible matrices} $\mathcal M_A \subset \mathcal M$ for which the REM \(T_W\) associated to the matrix \(W\in \mathcal M_A\) has the same combinatorics as the REM $T_{M_n}$ (see definition \ref{defn:admissible}). We say that two REMs, $T, T': X \to X$ with associated partitions $\mathcal A=\{A_i\}_{i=1}^N$, $\mathcal A'=\{A'_i\}_{i=1}^{N'}$ respectively, 
have the \textbf{\emph{same combinatorics}} if 
\begin{enumerate}
\item The cardinalities of the partitions $\mathcal A, \mathcal A'$ are equal.
\item For each \(i\), the polygons $A_i \in \mathcal A$ and $A'_i \in \mathcal A'$ have the same number of edges and edge directions, that is, they are the same up to changing edge lengths.
\item Two elements $A_i$ and $A_j$ in $\mathcal A$ meet along a common edge if and only if $A'_i$ and $A'_j$ share an edge in the corresponding position. 
\end{enumerate}
See Figure \ref{fig:same-comb} for an example of two REMs with the same combinatorics. 
The admissibility condition on $\mathcal M_A \subset \mathcal M$ is a set of linear equations in the eigenvectors of \(M\) (see definition \ref{defn:admissible}).

Let \(W\in \mathcal M_A\) be written $W=M_{n_L} \dots M_{n_2} M_{n_1}$. Define $W_k = M_{n_k} \dots M_{n_2} M_{n_1}$ for $1 \leq k \leq L$. When each REM $T_{W_k}$ has the same combinatorics as $T_{W}$ for every $k=1,\dots, L $ we call \(T_W\) a \textbf{\emph{multistage REM}} (see Definition \ref{defn:multi}). We use \({\mathcal M}_R\subset{\mathcal M}_A\) to denote the subset of admissible matrices which produce multi-stage REMs.

\begin{figure}
  \begin{center}
    \includegraphics[scale=1.1]{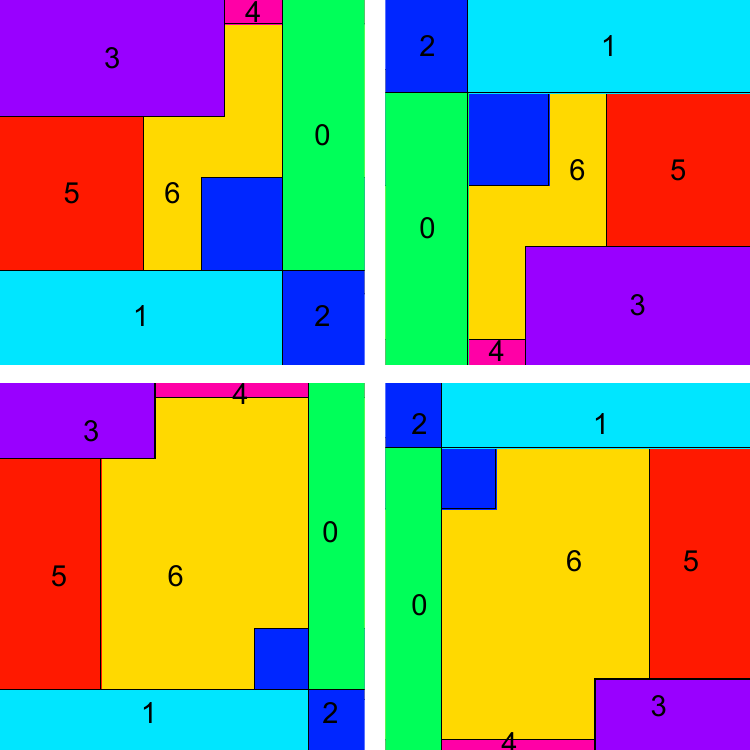}
\caption{The REM described by the top panels has the same combinatorics as the REM described by the lower two panels.}
\label{fig:same-comb}
  \end{center}
\end{figure}


For a multistage REM we study the \textbf{\emph{first return map}} to one of the tiles in the partition and prove that it is affinely conjugate to the original map. This is known as a \textbf{\emph{renormalization scheme}}. Renormalization schemes are an essential tool in the study of long term behavior of dynamical systems.

\begin{defn}
Let \(T:X\to X\) be a map and \(Y\subset X\). The \textbf{\emph{first return map}} $\hat T|_Y$ maps a point $p\in Y$ to the first point in the forward orbit of $p$ lying in $Y$, i.e.
\[
\widehat T|_Y(p)=T^m(p) \quad \mbox{where $m=\min\{k \in \Z_+: T^k(p) \in Y\}$}.
\]
The notation \( T|_Y\) means the dynamical system \(T\) restricted to \(Y\).
\end{defn}
When \(X\) is a finite measure space and \(T\) a measure-preserving transformation, the Poincar\'e Recurrence Theorem \cite{p2017}
ensures that the first return map is well-defined for almost every point in the domain. 

\begin{defn}
A dynamical system \(T_1:X_1\to X_1\) has a \textbf{\emph{renormalization scheme}} if there exists a proper subset \(X_2\subset X_1\), a dynamical system \(T_2:X_2\to X_2\), and a homeomorphism \(\phi: X_1\to X_2\) such that
\[ \widehat T_1 |_{Y_1} = \phi ^{-1} \circ T_2 \circ \phi .\]
A dynamical system is \textbf{\emph{renormalizable}} or \textbf{\emph{self-induced}} if \(T_2=T_1\).
\end{defn} 

\subsection{Main Results}
The main focus of our paper is the development of a renormalization scheme for the multistage REM \(T_M\),
defined in Section \ref{sec:cons} below, for every \(M\in \mathcal M_A\).

\begin{thm}\label{thm:single}
Let \(M \in \S\) be a matrix and \(T_{M}\) the PV REM 
associated to the Galois lattice \(L_{\lambda}\) where \(\lambda\) is the leading eigenvalue of \(M\). Label the eigenvalues of \(M\) by \(\lambda_1,\lambda_2\) and \(\lambda_3\) in increasing order. Let \(Y \subset X\) be the tile in the partition corresponding to the rectangle \([1-\lambda_1, 1]\times[1- \lambda_2,1]\). The REM $T_\lambda$ is renormalizable, i.e., 
\begin{align*}
 \widehat T_{M}|_Y = \phi^{-1} \circ  T_{M} \circ \phi
\end{align*}
where $\phi: X \to Y$ is the affine map
\[
\phi:(x,y) \mapsto \left(\frac{x+\lambda_1-1}{\lambda_1}, \frac{y+\lambda_2-1}{\lambda_2}\right).
\]
\end{thm}

We next prove that multistage REMs are minimal and have a renormalization scheme with multiple steps.

\begin{thm}\label{thm:mult-mi}
Multistage REMs are minimal.
\end{thm}
\begin{thm}\label{thm:mult}
Let $W=M_{n_L} \cdots M_{n_2} M_{n_1}\in  \mathcal M_R$ and define $W_k = M_{n_k} \cdots M_{n_2} M_{n_1}$ for $1 \leq k \leq L$. The associated multistage REM is renormalizable, i.e., for each \(k\) there exists \(Y_k\subset X\) and an affine map \( \phi_k: Y_k \to X\) such that 
\begin{align*}
 \widehat T_{W_{k+1}}|_{Y_{k+1}} = \phi^{-1}_k \circ  T_{W_k} \circ \phi_k.
\end{align*}
Each affine map has the form 
\[
\phi_k:(x,y) \mapsto \left(\frac{x+x_k-1}{x_k}, \frac{y+y_k-1}{y_k}\right) 
\]
where \(x_k\) and \(y_k\) are the dimensions of the tile in the partition corresponding to the rectangle \([1-x_k,1]\times[1-y_k,1]\).
\end{thm}

We conjecture that the closure of the set of renormalizable multistage REMs is topologically a Cantor set.

\subsection{Background}

A DEM is an example of a discrete dynamical system which is a piecewise affine isometry. These systems have applications to the study of substitutive dynamical systems, outer billiards, and digital filters. Originally J. Moser proposed studying outer billiards as a toy model for celestial dynamics. In much the same manner, DEMs provide a toy problem for the study of Hamiltonian dynamical systems with nonzero field. See \cite{g2003} for a nice survey including many open questions related to 2-dimensional piecewise isometries.

Although the maps we study are locally translations, the sharp discontinuities produce a dynamical system with extremely rich long-term behavior. This complexity can even be seen in the 1-dimensional case of interval exchange transformations (IETs). We wish to classify points in the domain by the long-term behavior of their orbits. The domain of an affine isometry is subdivided into tiles on which the map is locally constant. Each point in a piecewise isometry can be classified by the sequence of tiles visited by the forward orbit of a point. The most basic question is to give an encoding for each point in terms of this sequence. While this problem is particularly challenging, there has been some success in classifying points into sets of points whose orbits are eventually periodic and those whose orbits are not periodic. Such a classification has been carried out successfully in a few particular cases, \cite{akt2001}, \cite{g2003}, \cite{lkv2004}, \cite{ah2013}, \cite{h2013} and \cite{s2014}. 

In each case the authors used the principle of \emph{renormalization} to study the dynamical system. Renormalization provides a way to understand the long-term behavior of a discrete dynamical system. Unfortunately for piecewise isometries in dimension 2 or higher there are no general methods for developing a renormalization scheme for a dynamical system. In the 1-dimensional case of the IET, G. Rauzy developed a general technique known as \emph{Rauzy induction} for finding a renormalization scheme for an IET \cite{r1979}. His method does not generalize to higher dimensions.

REMs were first studied by Haller who gave a minimality condition \cite{h1981}. Unfortunately this condition is extremely difficult to check in practice. Finding a recurrent REM was included as question \#19 in a list of open problems in combinatorics at the \emph{Visions in Mathematics} conference \cite{g2000}. Hooper developed the first renormalization scheme for a family of REMs parametrized by the square \cite{h2013}.  In \cite{s2014} Schwartz used multigraphs to construct polytope exchange transformations (PETs) in every dimension. He developed a renormalization scheme for the simplest case in which the corresponding multigraphs are bigons. The renormalization map is a piecewise 
M\"obius map. 

The topological entropy of a dynamical system gives a numerical measure of its complexity. For a dynamical system defined on a compact topological space the topological entropy is an upper bound for the exponential growth rate of points whose orbits which remain a distance \(\epsilon\) apart as \(\epsilon \to 0\) \cite{t2014}. The topological entropy gives an upper bound on the metric entropy of the dynamical system. In \cite{b2001} J. Buzzi proved that the topological entropy is zero for piecewise isometries defined on a finite union of polytopes in \(\R^d\) which are actual isometries on the interior of each polytope. The REMs we study in this paper are examples of such systems and as a consequence have zero topological entropy. However when the domain is not a union of polytopes the techniques in \cite{b2001} must be modified. We expect that our technique for constructing domain exchange maps produces dynamical systems with zero topological entropy but have not proved this.

Throughout this paper we make extensive use of the connection between non-negative integer matrices and \emph{Perron numbers}. A Perron number is a positive real algebraic integer \(\lambda\) which is strictly larger than the absolute value of any of its Galois conjugates. In \cite{l1984} it was proven that for every Perron number \(\lambda\) there exists a non-negative integer matrix $M$ which is \emph{irreducible} (i.e. $M^k$ is positive for some power $k$) and has \(\lambda\) as a leading eigenvalue. 

In this paper we use algebraic properties of a subset of Perron numbers known as \emph{Pisot-Vijayaraghavan numbers} or \emph{PV numbers} to find REMs which are renormalizable. A PV number is a positive real algebraic integer whose Galois conjugates lie in the interior of the unit disk. We use cut-and-project sets associated to PV numbers to produce DEMs. Cut-and-project sets were introduced in
\cite{m1994} and further studied in \cite{l1996}.

Our proof of the renormalization schemes in this paper rely on algebraic properties of PV numbers. In two recent works monoids of matrices were discovered whose leading eigenvalues are PV numbers (\cite{ai2001} and \cite{ad2015}). The authors called these matrices \emph{Pisot matrices}. We find a new monoid of Pisot matrices with an infinite generating set.

The techniques we use in this paper are influenced by \cite{k1992} and \cite{k1996}. These works focused on self-similar tilings of the plane whose expansion constant is a complex Perron number. Unlike the tiles in our DEMs, the tiles in \cite{k1996} have a fractal boundary. Our construction of DEMs also share similarities with the Rauzy fractal \cite{r1982}.

\section{Constructing minimal DEMs with cut-and-project sets}\label{minimal-DEMs}

\subsection{Definition}
Let \(X\) be a smooth Jordan domain in \(\R^2\) and \(L\) a lattice in \(\R^3\) such that \(\Lambda = \Lambda(X,L)\) is a cut-and-project set: 
$\Lambda=\{p\in L|~ \pi_{xy}(p) \in X\}$. We construct a DEM on \(X\) by projecting a dynamical system on \( \Lambda \) onto the window \(X\). Projection onto the \(z\)-coordinate gives an ordering of the points in \(\Lambda\). Order the points in $\Lambda$ by increasing $z$-coordinate: $\Lambda=\{\dots,\mathbf x_{-1},\mathbf x_0,\mathbf x_1,\dots\}$. Let \(\widetilde{T}:\Lambda \to \Lambda\) be the dynamical system defined by
\begin{align*}
\widetilde{T}(\mathbf x_i) & = \mathbf x_{i+1}.
\end{align*}

Consider the set of steps in the lattice walk
\begin{align*}
\mathcal E= \{ \widetilde{T}(\mathbf x) -\mathbf x : \mathbf x \in \Lambda\}.
\end{align*}
Since $L$ is a lattice, \(\mathcal E\) is a finite set. Suppose there are \(N+1\) vectors in \(\mathcal E\) and label them by \(\mathcal E = \{ \eta_0, \eta_1,\ldots, \eta_N\}\). Projection onto the \(z\)-coordinate induces an order on \(\mathcal E\). We assume that \(\mathcal E\) is indexed so that
\[\pi_{z}(\eta_0)< \pi_{z}(\eta_1)< \dots < \pi_{z}(\eta_{N}).\]
Define \(  V= \{ v_i = \pi_{xy} (\eta) ~:~ \eta \in \mathcal E \}\). The DEM \(T:X\to X\) is defined by 
\[ T(p) =p+v_i  \text{ with } i = \min\{ 0, \ldots, N ~:~ v_j \in V \text{ and } p + v_j \in X\}. 
\]
Note that \(T\) is well-defined and bijective on \(X\). The map $T$ is a piecewise translation on $X$.

The DEM induces a partition of \(X\) into subdomains \(\{A_i\}_{i=0}^N\) for which \(T(p) = p + v_i\) for all \(p\in A_i\). Likewise \(T^{-1}\) induces a partition \(\{B_i\}_{i=0}^n\) for which \(T^{-1}(p) = p -v_i\) for all \(p\in B_i\). Note that 
\[ X = \bigcup_{k=0}^N A_k = \bigcup_{k=0}^N B_k\]
and \(A_k =B_k + v_k\), verifying that \(T\) is a DEM. The subdomains are not necessarily connected. However, each connected component of a subdomain is bounded by a smooth Jordan curve as long as \(X\) is a smooth Jordan domain. 

In Figure \ref{fig:3dpath} we show both the lattice walk \(\widetilde T\) and the resulting DEM \(T\).

\begin{figure}[h]
  \begin{center}
    \includegraphics[width=\textwidth]{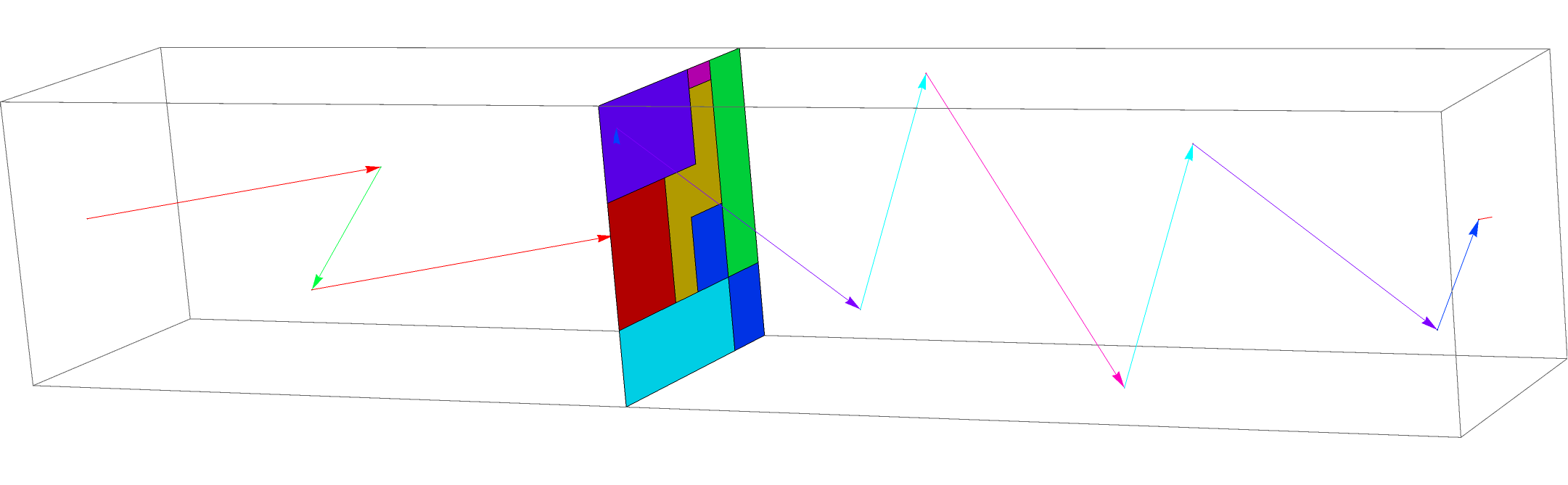}
\caption{Lattice walk in \(\Lambda(X,L)\) and the partition associated to the DEM on \( X\). Each colored region in the partition is translated by the projection of the step in the lattice walk, with the same color, onto the \(xy\)-plane.}
\label{fig:3dpath}
  \end{center}
\end{figure}

For a dynamical system $T: X \to X$, the \emph{orbit} of \(p\) is the set $O(p)=\{T^j(p)~|~j\in\Z\}$. We also define 
$O^{k+}(p)=\{T^j(p)| \ j\in \Z, 0 \leq j \leq k\}$ the \(k\)-th \textit{forward orbit} of $p$, and $O^+(p)=\{T^j(p)~|~j\ge0\}$ the forward orbit.

\subsection{Vertical flow}

Let $\T^3 = \R^3/L.$ We can consider $X$ as a subset of $\T^3$: the inclusion map $\i:X\to\T^3$ is injective by our conditions on $L$.
On $\T^3$ the \emph{vertical linear flow} is defined by $\Phi_t((x,y,z)) = (x,y,z+t)\bmod L$ for $t\in\R$.

By Weyl's Equidistribution Theorem (see e.g. \cite{Stein}), the vertical flow is equidistributed on $\T^3$ in the following sense. 
Take any open set $\Omega$
in the image of the $xy$-plane in $\T^3$, and a point $x\in\Omega$. 
The iterates of the first return map to $\Omega$ of the vertical flow, when applied to $x$, are equidistributed in $\Omega$. 

So to prove Theorem \ref{minimality} above, it suffices to establish the following result.

\begin{thm}
$T$ is conjugate to the first return map to $X$ of the vertical linear flow $\Phi$, that is
$\i(T(p)) = \Phi_\tau(\i(p))$ where
$$\tau = \inf\{t>0~|~ \Phi_t(\i(p))\in \i X\}.$$
\end{thm}

\begin{proof}
The vertical linear flow on $\T^3$ lifts to the vertical flow on $\R^3$. 
Consider all translates of $X\in\R^2\subset\R^3$ by lattice translations in $L$. 
Each of these intersects $X\times\R$ in some (possibly empty) subset.
Order those with nonempty intersections by their $z$-coordinate.
By construction the translates $\eta_0+X, \dots,\eta_N+X$ are the first $N+1$ such translates, and the projections
to $\R^2$ of these cover $X$. 
\end{proof}

\subsection{PV REMs}\label{sec:cons}

\begin{figure}[h]
\centering
\includegraphics[scale=1]{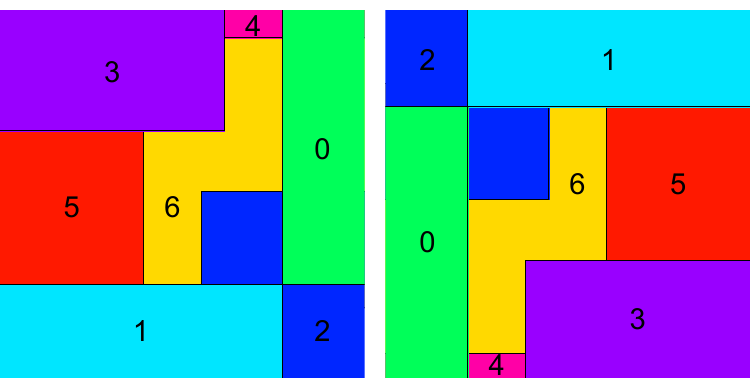}\\
\caption{The two partitions associated to the REM \(T_{M_6}\)}
\label{fig:M6}
\end{figure}

We explain here the details of the REM construction when \(X=[0,1]\times[0,1]\) and \(L\) is the Galois embedding of \(\Z[\lambda]\) 
where \(\lambda\) is a certain family of PV numbers. 
Define for each $n\ge 6$ a polynomial
\[
q_n(x)=x^3-(n+1)x^2+nx-1.
\]

\begin{lemma}\label{lem:roots}
The polynomial \(q_n\) has three real roots, $\lambda_1, \lambda_2$ and $\lambda_3$, which satisfy the inequalities $0<\lambda_1<\lambda_2<1<\lambda_3$. 
\end{lemma}

\begin{proof}
The discriminant of $q_n$ is
\[ D(n) = n^4-6 n^3+7 n^2+6 n-31.\]
It has two real roots  \(n=1/2 (3 + \sqrt{13 + 16 \sqrt{2}})\) and \(1/2 (3 - \sqrt{13 + 16 \sqrt{2}})\). Thus for \(n\geq 6\) the discriminant is strictly positive and we find \(q_n\) has three distinct real roots. 

Since 
\[
\lambda_1 \lambda_2 \lambda_3=1 \quad \mbox{and} \quad \lambda_1+\lambda_2+\lambda_3=n+1
\]
it follows that  \(\lambda_3>1\) and \(\lambda_1<1\). However, \(\lambda_3<n+1\) and so \(\lambda_1+\lambda_2>0\). This implies \(\lambda_2>0\). The product of the three roots is one which implies that \(\lambda_1>0\). 

It remains to show that \(\lambda_2<1\). Evaluating \(q_n\) and its derivative at \(0\) and \(1\) gives
\[
q_n(0)=-1, \quad q_n'(0)=n, \quad q_n(1)=-1 \quad \text{and} \quad q_n'(1)=1-n . 
\]
We find that $q_n(x)$ has two roots between $0$ and $1$ and conclude that $0<\lambda_1<\lambda_2<1$.
\end{proof}

Note that \(q_n\) is the characteristic polynomial of the matrix
\[ M_n = \begin{bmatrix} 0 & 1 & 0 \\ 0 & 0 & 1 \\ 1 & -n & n+1\end{bmatrix}.\]
Let \(T_{M_n}:X \to X\) be the PV REM associated to the Galois embedding of the roots of \(q_n\). The two partitions associated to the REM \(T_{M_6}\) are shown in figure \ref{fig:M6}. When \(L\) has this form there are seven possible steps in the lattice walk \(\mathcal E_n\). It is convenient to identify points in \(L\) by their representation in \(\Z^3\), i.e., if \((a,b,c) \in \Z^3\) then 
\[ \pi_{xy}(a,b,c) = (a+b \lambda_1 + c \lambda_1^2,a+b \lambda_2 + c \lambda_2^2) \text{ and } \pi_z(a,b,c) = a+b\lambda_3+c\lambda_3^2.\]
Using this representation the vectors in \(\mathcal E_n\) are
\begin{align}\label{lattice-walk-vectors}
\begin{aligned}
\eta_0 & =(-1,1,0), \quad \eta_1=(0,1,0), \quad \eta_2=\eta_0+\eta_1=(-1,2,0) \\ 
\eta_3 & =(1,-3,1), \quad \eta_4=\eta_0+\eta_3=(0,-2,1), \\
\eta_5 & =\eta_1+\eta_3=(1,-2,1), \quad \mbox{and} \quad \eta_6  =\eta_0+\eta_1+\eta_3=(0,-1,1). 
\end{aligned}
\end{align}
Theorem \ref{thm:comb} establishes that the steps in the lattice walk are independent of \(n\) and as a consequence we set \(\mathcal E_n=\mathcal E\).

The partition associated to the REM \(T_{M_n}\) is constructed as follows. A visual depiction of the construction is shown in figure \ref{fig:par}. Define the projections onto the \(xy\)-plane of the translation vectors in \(\mathcal E\) by
\[ V_n= \{ v_i = \pi_{xy}(\eta_i), \text{ for } i = 0, 1, \ldots 6\}.\]
Note that \(V_n\) depends on \(n\) since the projection \(\pi_{xy}\) is a function of the roots of \(q_n\).

For a vector \(v\in \R^2\) let \(f_v\) be the translation \(f_v(x)= x+v\) for \(x\in \R^2\). We define the partition \(\mathcal A = \{A_k\}_{k=0}^N \) of \(X\) associated to \(T_{M_n}\) inductively as follows:
\begin{equation}
A_0 = f^{-1}_{v_0}(X) \cap X \quad \mbox{and} \quad  A_k=(f^{-1}_{v_k}(X) \cap X ) \setminus \bigcup_{j=0}^{k-1} A_j \quad \text{for} \quad k >0. 
\label{const}
\end{equation}
For a point \(x\) in the interior of a tile in the partition \(A_k\) the dynamical system is defined by
\[
T_{M_n}|_{A^\circ_k}(x)=f_{v_k}(x)= x+ v_k.
\]

\begin{figure}[h]
\begin{center}
\includegraphics[scale=0.8]{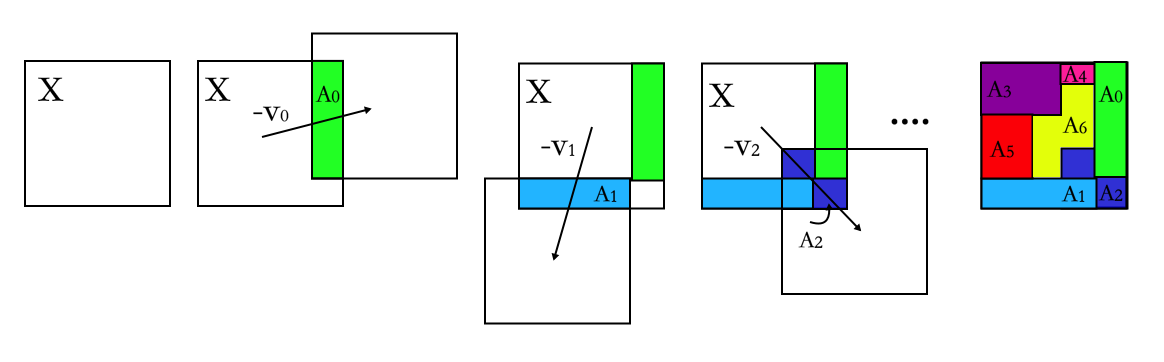}
\caption{The steps in the construction of the partition $\mathcal A$ associated to the REM \(T_{M_n}\) and the resulting partition.}
\label{fig:par}
\end{center}
\end{figure}

Each tile in the partition \(\mathcal A\) is a rectilinear polygon (refer to the example in Figure \ref{fig:M6}) and can be written as a disjoint union of rectangles. We use the standard notation for a rectangle
\[  [a,b] \times [c,d] = \{ (x,y) \in \R^2 : a \leq x \leq b \text{ and } c \leq y \leq d\}.\]
Recall that \(\lambda_1\) and \(\lambda_2\) are roots of the polynomial \(q_n(x)\) with \(0< \lambda_1< \lambda_2<1\). 
The tiles are as follows
\begin{align}\label{eq:partition}
\begin{aligned}
 A_0 &=[1- \lambda_1,1] \times [1-\lambda_2,1] \\
A_1 &= [0, 1- \lambda_1]\times [0, 1- \lambda_2]  \\
A_2 &= \left([1-2\lambda_1,1-\lambda_1] \times[1-\lambda_2,2-2\lambda_2] \right) \cup \left( [1-\lambda_1,1]\times[0,1-\lambda_2] \right)  \\
A_3 &=[0,3\lambda_1-\lambda_1^2] \times [-1+3 \lambda_2-\lambda_2^2,1] \\
A_4& = [3\lambda_1-\lambda_1^2, 1-\lambda_1] \times [2\lambda_2-\lambda_2^2,1] \\
A_5&=[0, 2\lambda_1-\lambda_1^2] \times[ 1- \lambda_2, -1+3\lambda_2-\lambda_2^2 ] \\
A_6 &=\big ( [1-2\lambda_1,3\lambda_1-\lambda_1^2] \times[2-2\lambda_2,-1+3\lambda_2-\lambda_2^2] \big) \\
& \cup \big ( [2\lambda_1 - \lambda_1^2, 1-2 \lambda_1]\times[1-\lambda_2,-1+3 \lambda_2-\lambda_2^2] \big )\\
&\cup\big( [3\lambda_1-\lambda_1^2,1-\lambda_1]\times[2-2\lambda_2,2\lambda_2-\lambda_2^2] \big) .
\end{aligned}
\end{align}


\section{Analysis of the PV REM \(T_{M_6}\) and its Renormalization}
Before analyzing the general case, we give a detailed description of the PV REM \(T_{M_6}\) in which the Galois lattice \(L_{\lambda}\) is determined by the polynomial \(q_6(x)=x^3-7x^2+6x-1\). 

Let $V=\{v_i\}_{i=0}^6$ be the set of translation vectors of  the REM $T_{M_6}$ where $v_i=\pi_{xy}(\eta_i)$ for $\eta_i \in \mathcal E$ listed in Lemma \ref{lem:comb6}. We obtain the REM $T_{M_6}: X \to X$ defined on the partition $\{A_i\}_{i=0}^{6}$ as shown in Figure \ref{fig:M6}.

\begin{lemma}{\label{lem:comb6}}
Let $\mathcal E=\{\eta_i\}_{0}^6$ where the $\eta_i$ are defined in (\ref{lattice-walk-vectors}).
The set of translation vectors of $T_{M_n}$ are $\{\pi_{xy}(\eta_i)\}_{i=1}^6$ for $n= 6$.
\end{lemma}

\begin{proof}
The characteristic polynomial of the matrix 
\[
M_6=\begin{bmatrix}
0 & 1 & 0 \\
0 & 0 & 1 \\ 
1 & -6 & 7
\end{bmatrix}.
\]
is $q_6(x)=x^3-7x^2+6x-1$. By Lemma \ref{lem:roots}, the polynomial $q_n(x)$ has three roots $\lambda_1, \lambda_2$ and $\lambda_3$ with $0<\lambda_1<\lambda_2<1<\lambda_3$. The eigenvector $\xi_i$ of $M_6$ associated to $\lambda_i$ is $(1, \lambda_i, \lambda^2_i)$ for $i=1,2$ and $3$. 

By direct computation, we find that the seven vectors \(\eta_0, \eta_1, \ldots, \eta_6\) are the seven solutions for vectors in $\Z^3$ of the following inequalities
\begin{align*}
-1&< v \cdot \xi_1 <1 \\
-1&< v \cdot \xi_2 <1 \\
0&< v \cdot \xi_3  <31.
\end{align*}
The first two equations ensure that the projection of each step of the lattice walk in $\Z^3$ is a translation vector in the REM. The third equation ensures that these are the first seven vectors in \(\mathcal E\) which define a partition of the unit square. The set of real solutions to the above inequalities is a convex polytope in \(\R^3\) which contains exactly seven integer points. Each solution corresponds to a permissible step in the lattice walk on \(\Lambda_X\). 
\end{proof}

\begin{thm}\label{thm:M6}
Let $Y=A_0$. The first return map $\hat T_{M_6}|_{Y}$ to the set $Y$ is conjugate to $T_{M_6}$ by the affine map $\psi: Y \to X$ given by
\[
\phi_n(x,y)=\left(\frac{x+\lambda_1-1}{\lambda_1}, \frac{y+\lambda_2-1}{\lambda_2}\right).
\]
where $0<\lambda_1<\lambda_2<1$ are the smaller eigenvalues of the matrix $M_6$. 
\end{thm}

Theorem \ref{thm:M6} is a particular case of Theorem \ref{thm:single} whose proof is given in Section \ref{prf:single}. In the Appendix we give a computational proof of Theorem \ref{thm:M6} and a symbolic encoding of the partition of $Y$ induced by the first return map $\hat T_{M_6}|_Y$. 

\begin{figure}[h]
\centering
\includegraphics[width=0.48\textwidth]{M6}
\includegraphics[width=0.48\textwidth]{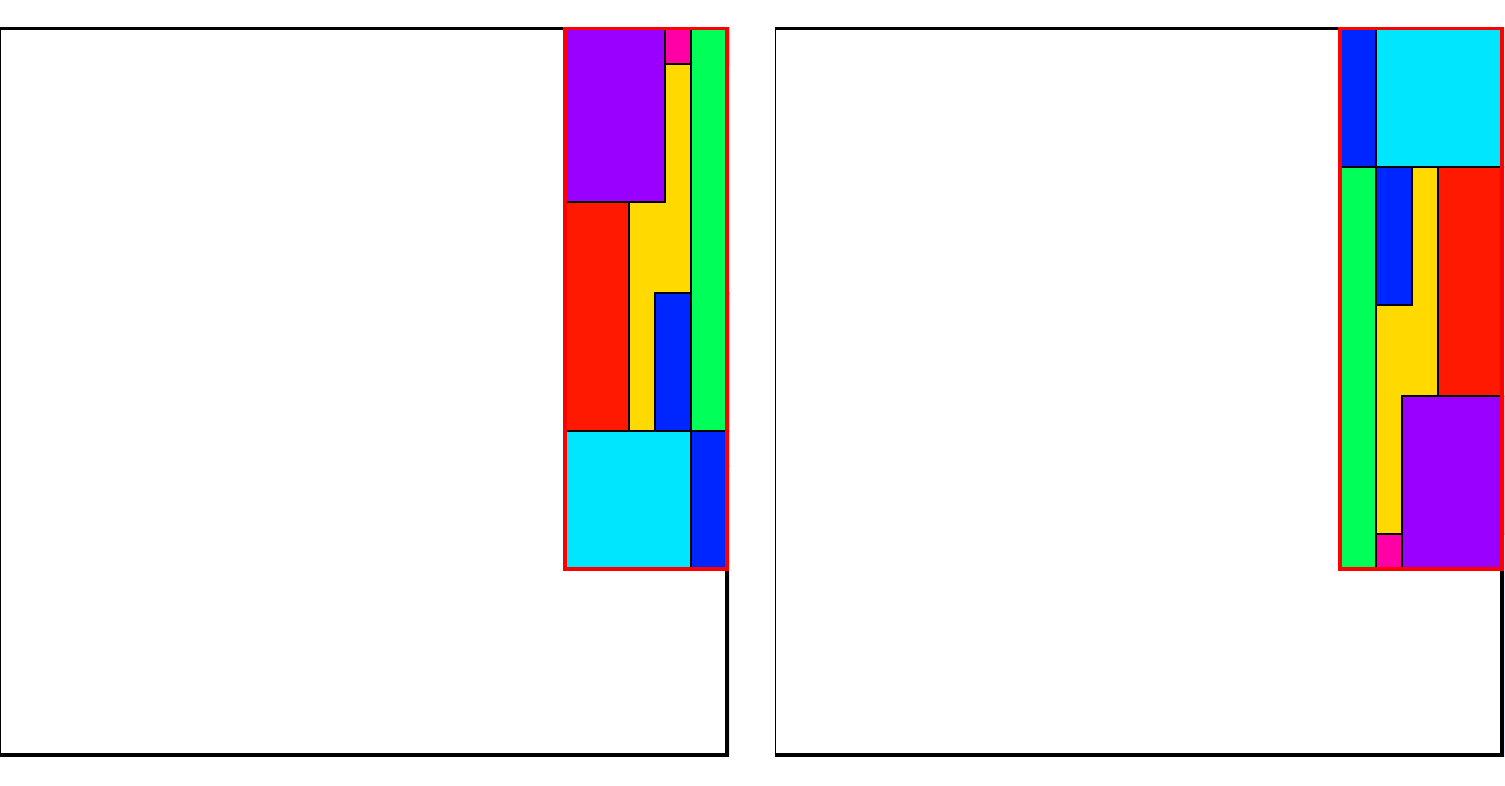}
\caption{The REM \(T_{M_6}\) and the partition induced by the first return map $\hat T_{M_6}|_{Y}$  to \(Y=A_0\).}
\label{fig:first6}
\end{figure}

\section{The renormalization scheme for PV REMs}
\subsection{Analyzing the lattice walk for \(M_n\in\S\)}

Let \(T_{M_n}\) be the PV REM constructed from a matrix \(M_n\in \S\) using the method outlined in section \ref{sec:cons}. 
Let \(L=L_{\lambda}\) be the associated Galois lattice. In this section, we analyze the dynamical system \(\widetilde{T}\) on \(L\) and prove 
\begin{thm}\label{thm:comb}
Lemma \ref{lem:comb6} holds for all $n \geq 6$.
\end{thm}

There are a number of steps in the proof. The first step is proving a more refined version of Lemma \ref{lem:roots}.
\begin{lemma}\label{lem:monotone}
Label the roots of $q_n$ \(\lambda_1^n,\lambda_2^n,\lambda_3^n\) with \(0<\lambda_1^n<\lambda_2^n<1<\lambda_3^n\). Then \(\lambda_2^n\) and \(\lambda_3^n\) are monotonically increasing functions of \(n\) while \(\lambda_1^n\) is monotonically decreasing as a function of \(n\). Moreover we have the following inequalities 
\begin{align*}
n&<\lambda_3^n<n+1 \\
1-\frac{1}{n-3}&<\lambda_2^n<1-\frac{1}{n-2} \\
\frac{1}{n-1}&<\lambda_1^n< \frac{1}{n-2}.
\end{align*}
\end{lemma}

\begin{proof}
The polynomial is cubic and therefore changes sign at most three times. We find three disjoint intervals in which \(q_n\) changes sign. Since the polynomial is cubic each root must lie in one of these intervals.
\begin{align*}
 q_n(n)= -1<0 \qquad &\text{and} \qquad q_n(n+1)=n^2+n-1>0, \\
 q_n\left(1-\frac{1}{n-2} \right) = -\frac{1}{(n-2)^3} <0 \qquad &\text{and} \qquad
q_n\left(1-\frac{1}{n-3} \right) =\frac{11-7n+n^2}{(n-3)^3}>0, \\
q_n \left(\frac{1}{n-1}\right)=\frac{3-2n}{(n-1)^3}<0 \qquad &\text{and} \qquad q_n\left(\frac{1}{n-2} \right)  = \frac{11-7n+n^2}{(n-2)^3}>0.
\end{align*}
This establishes the desired inequalities. The monotonicity of the roots can be verified from the inequalities by inspection.
\end{proof}

Recall the definitions of $\eta_0,\eta_1,\eta_3$ from (\ref{lattice-walk-vectors}). Since $\eta_0, \eta_1$  and $\eta_3$ are independent over $\Z$, every element $\omega \in \Z^3$ can be written as  
\[
\omega=a \eta_0 + b \eta_1 + c \eta_3, \qquad \text{for $a,b,$ and $c \in \Z$}.
\]
The following lemma is an important step in the proof of Theorem \ref{thm:comb}.

\begin{lemma}{\label{lem:step}}Each element of $\mathcal E_n$ is a nonnegative linear combination of $\eta_0,\eta_1,\eta_3$.
\end{lemma}

\begin{proof}
Note that $\pi_{z}(\eta_i)>0$ for $i=0,\dots, 3$ and $\eta_2=\eta_0+\eta_1$. Here we discuss all possible cases of $\omega \in \Z^3$ such that 
\begin{enumerate}
\item $\pi_{xy}(\omega) \in (-1,1)^2$ which ensures that $\pi_{xy}(\omega)$ is a translation vector on $X$.
\item $0<\pi_{z}(\omega)<\pi_{z}(\eta_i)$ for some $i=0,1$ and $3$. 
\end{enumerate}
\vskip 1em

\textbf{Case 1: } $\omega=a\eta_0-b\eta_1$ for positive integers $a$ and $b$. Suppose that the vector
\[
\omega=a\eta_0-b\eta_1= (-a,a-b,0)
\]
has $\pi_{z}(\omega)>0$ and $\pi_{xy}(\omega) \in (-1,1)^2$. The y-component of the projection $\pi_{xy}(\omega)$ is 
\[
-a+(a-b) \lambda_2
\] 
where $\lambda_2$ is the second largest eigenvalue of matrix $M_n$ for some $n$. By assumption, we have
\[
-1 < -a + (a-b) \lambda_2<1.
\]
It follows that 
\[
 \frac{-1+a}{\lambda_2}< a-b <\frac{1+a}{\lambda_2}.
\]
By Lemma \ref{lem:monotone} and $\dfrac{2}{3}<\lambda_2 <1$
\[
-1+a< \frac{-1+a}{\lambda_2}< a-b <\frac{1+a}{\lambda_2} < \frac{3}{2}(1+a).
\]
Then, we can conclude that 
\[
-2<b<1
\]
which contradicts to the assumption that $a,b \geq 1$. 

This argument also shows that if $\omega=a\eta_0-b\eta_1$ with $a,b \in \Z^2$ positive, $\pi_{xy}(-\omega) \notin (-1,1)^2$.\\

\textbf{Case 2:} $\omega = c\eta_3-b\eta_1$ for positive integers $b$ and $c$.
Note that 
\[
\omega=c\eta_3-b\eta_1 = c(1,-3,1)-b(0,1,0) = (c, -3c-b, c).
\]
Consider the y-component of $\pi_{xy}(\omega)$: we have 
\begin{align*}
c-(3c+b)\lambda_2 + c \lambda_2^2  \ & \leq c(1+\lambda_2^2)-(3c+1)\lambda_2\\
& \leq 2 c - \dfrac{3}{4}(3c+1)\\
& \leq -\dfrac{1}{4} c - \frac{3}{4} \leq -1.
\end{align*}
It follows that $\pi_{xy}(\omega) \notin (-1,1)^2$ for all $c\eta_3-b\eta_1$ with $v,t \in \Z_+$.  Similarly, $\pi_{xy}(\omega) \notin (-1,1)^2$ for all $\omega=b\eta_1 - c\eta_3$ with positive integers $b$ and $c$.\\

\textbf{Case 3: } $\omega = c\eta_3 - a \eta_0$ for positive integers $a$ and $c$. Note that
\[
c\eta_3-a\eta_0 = c(1,-3,1)-a(-1,1,0) = (a+c, -3c-a, c).
\]
Consider the $x$-coordinate of the projection $\pi_{xy}(\omega)$. By Lemma \ref{lem:monotone}$, \lambda_1\leq 1/4$ and we have
\begin{align*}
(a+c)-(3c+a) \lambda_1 + c \lambda^2_1 & \geq  (a+c)-(3c+a) \frac{1}{4} + c \lambda^2_1\\
& \geq \frac{1}{4} c +\frac{3}{4}a + c \lambda^2_1  \geq 1.
\end{align*}
Therefore, $\pi_{xy}(\omega) \notin (-1,1)^2$ for all $\omega=c\eta_3 - a \eta_0$ with integers $a, c \geq 1$. Similarly, if $\omega=a\eta_0 - c\eta_3$ with positive coefficients $a,c$, then the $x$-coordinate of $\pi_{xy}(\omega)$ is less than $-1$.  \\

\textbf{Case 4: } $\omega=c\eta_3+a\eta_0-b\eta_1$ with positive integers $a, b$ and $c$.  Consider the y-component of $\pi_{xy}(\omega)$
\[
\pi_{xy}(c\eta_3+a\eta_0-b\eta_1)_y=\pi_{xy}(c\eta_3-b\eta_1)_y + a\pi_{xy}(\eta_0)_y
\]
By Case 2, $\pi_{xy}(c\eta_3-b\eta_1)_y \leq -1$ for all $b,c \in \Z_+$. Moreover, $\pi_{xy}(\eta_0)_y <0$. Thus, there is no possible $\omega=c\eta_3+a\eta_0-b\eta_1$ with $\pi_{xy}(\omega) \in (-1,1)^2$. For the same reason, $\pi_{xy}(-\omega) \notin (-1,1)^2$. \\ 

\textbf{Case 5: } $\omega=c\eta_3-a\eta_0+b\eta_1$ for $a,b,c \in \Z_+$.  Consider the $x$-component $\pi_{xy}(\omega)_x$ of the projection $\pi_{xy}(\omega)$ given as
\[
\pi_{xy}(c\eta_3-a\eta_0+b\eta_1)_x=\pi_{xy}(c\eta_3-a\eta_0)_x + b\pi_{xy}(\eta_1)_x.
\]
In Case 3, we show that  $\pi_{xy}(c\eta_3-a\eta_0)_x \geq 1$ for all positive integers $a$ and $c$. Since 
\[
\pi_{xy}(\eta_1)_x >0
\] 
we have $\pi_{xy}(\pm\omega) \notin \mathcal (-1,1)^2$.\\

\textbf{Case 6: } $\omega=c \eta_3 - a\eta_0 - b\eta_1$ for positive integers $a,b,c$.
Since $\omega=(a+c, -3c-a-b, c)$
\[
\pi_{xy}(\omega)=(a+c-(3c+a+b)\lambda_1+c \lambda_1^2, \ a+c-(3c+a+b)\lambda_2+c \lambda_2^2).
\]
We consider the difference $|\pi_{xy}(\omega)_y-\pi_{xy}(\omega)_x|$ which is
\begin{align*}
|\pi_{xy}(\omega)_y - \pi_{xy}(\omega)_x| & = |-(3c+a+b)(\lambda_2-\lambda_1)+c(\lambda_2^2-\lambda_1^2)|\\
&=|(\lambda_2-\lambda_1)[c(\lambda_1+\lambda_2-3)-a-b]|
\end{align*}
By Lemma \ref{lem:monotone}, we have $0\leq \lambda_1\leq 1/5$ and $3/4 \leq \lambda_2 \leq 1$ where $\lambda_2$ is the second largest eigenvalue for matrix $M_n$ with $n\geq 7$. Therefore, 
\begin{align*}
|\pi_{xy}(\omega)_y - \pi_{xy}(\omega)_x| & \geq \frac{1}{2}\ |c(\lambda_1+\lambda_2-3)-a-b|\\
& \geq \frac{1}{2} \ |-2c -a-b |. 
\end{align*}
Since $a, b ,c \geq 1$ are integers, it means that $|\pi_{xy}(\omega)_y - \pi_{xy}(\omega)_x| \geq 2$. It follows that $\pi_{xy}(\omega)_x$ and $\pi_{xy}(\omega)_y$ cannot be in the interval $(-1,1)$ at the same time. It follows that $\pi_{xy}(\omega) \notin (-1,1)^2$ for any positive integer $a,b,c$. Moreover, $\pi_{xy}(-\omega) \notin (-1,1)^2$. \\

\textbf{Case 7. } $\omega=a\eta_0+b\eta_1$ for non-negative integers $a$ and $b$ with $a\geq 2$ or $b\geq 2$. We compute the case when $a=2$. Then $\omega=2\eta_0 = (-2, 2,0)$ which implies that the x-coordinate of $\pi_{xy}(\omega) \notin (-1,1)$ by Lemma \ref{lem:monotone}. Similarly, when $b=2$ we compute $\omega=2\eta_1=(0,2,0)$ and the y-coordinate of the projection $\pi_{xy}(\omega)$ is not in the interval $(-1,1)$.

Therefore, we remain to check the case when $a+b\geq 3$ for non-negative integers $a$ and $b$. We have the vector $\omega=a\eta_0+a\eta_1 = (-a,a+b,0)$. Therefore 
\[
-a + \frac{a+b}{n-1} \leq \pi_{xy}(\omega)_x \leq -a + \frac{a+b}{n-2}
\]
\[
b-\frac{a+b}{n-3}\leq \pi_{xy}(\omega)_y \leq b-\frac{a+b}{n-2}
\]
so that 
\begin{eqnarray*}
\pi_{xy}(v)_y - \pi_{xy}(v)_x &\geq&  (a+b)(1-\frac{1}{n-3}-\frac{1}{n-2})\\
&\geq& 3(1-\frac{2}{n-3}) \geq 2 \quad \mbox{for $n\geq 9$}.
\end{eqnarray*}

When $n=7$, 
\[
-a+\dfrac{1}{2}\leq \pi_{xy}(\omega)_x \leq -a+\dfrac{3}{5}
\]
so that if $\pi_{xy}(\omega)_x \in (-1,1)$, then $a$ must be $0$ or $1$. It means that $b=3$ or $b=2$ respectively. However, 
\[
b-\frac{3}{4}\leq \pi_{xy}(\omega)_y \leq b-\frac{3}{5}.
\]
For either case, $\pi_{xy}(\omega)_y > 1$. The proof of the case $n=8$ is the same.
\end{proof}

\begin{proof}[Proof of Theorem \protect{\ref{thm:comb}}] 
Recall that \(\mathcal E_n\) is defined to be a set of steps in the lattice walk \(\widetilde T:\Lambda(X,L)\to \Lambda(X,L)\). By Lemma \ref{lem:step} every vector in \(\mathcal E_n\) is a non-negative linear combination of $\eta_0$, $\eta_1$ and $\eta_3$. We show that the seven vectors in \(\mathcal E_n\) with the smallest projections under \(\pi_z\) are sufficient to describe all steps in the lattice walk \(\widetilde T\). Moreover, Lemma \ref{lem:step} establishes that the seven vectors in Equation \ref{lattice-walk-vectors} are exactly the seven shortest vectors in \(\mathcal E_n\).

In Equation \ref{eq:partition} we construct the partition \(\mathcal A=\{A_i\}_{i=0}^6\) with translation vectors \(v_i=\pi_{xy}(\eta_i)\). Applying the inequalities from Lemma \ref{lem:monotone} one can verify that \(\mathcal A\) gives a partition of \(X\) into seven rectilinear polygons with disjoint interiors. Let \(p\in \Lambda(X,L)\) and \(A_i\) the tile with \(\pi_{xy}(p)\in A_i\). Then \(\pi_{xy}(p) + v_i\in  X\) since \(X\) overlaps with \(X+v_i\) for each \(i=0,1,\ldots,6\).  It follows that \(p + \eta_i\in \Lambda(X,L)\) and therefore $\eta_i$ is a valid step in the lattice walk.  Since \(p\) is an arbitrary point in \(\Lambda(X,L)\) we conclude that the vectors in $\mathcal E_n=\{\eta_i\}_{i=0}^6$ are sufficient to define all of the steps in the lattice walk in \(\Lambda(X,L)\).
\end{proof}

\subsection{Proof of Theorem \ref{thm:single}}\label{prf:single}
Fix $n\ge 6$ and consider the REM $T_{M_n}: X\to X$. Let $Y$ be the rectangle $A_0 \in \mathcal A$. It is sufficient to compute the first return map for the lattice walk \(\widetilde T\) because the lattice is dense in \(X\) and points which are sufficiently close in \(X\) have the same sequence of translation of vectors for finite time.

Define 
\[
\Lambda_X= \Lambda(X, L) \ \mbox{and} \ \Lambda_Y= \{(x,y,z) \in \Z^3~|~ \pi_{xy}(x,y,z) \in Y\}.
\] 
Since $Y\subset X$, we have $\Lambda_Y \subset \Lambda_X$. Let $(a,b,c)$ be a lattice point in $\Lambda_X$. Consider the map $\Psi$ defined by
\[
\Psi: \begin{bmatrix}
a\\ b\\ c 
\end{bmatrix} \mapsto 
(M_n)^t\begin{bmatrix}
a\\ b\\ c 
\end{bmatrix} +
\begin{bmatrix}
1 \\ -1 \\ 0
\end{bmatrix}.
\]

We show that $\Psi$ maps $\Lambda_X$ to $\Lambda_Y$. Then 
\[
\pi_{xy} \circ \Psi \left(  \begin{bmatrix}
a\\ b\\c
\end{bmatrix}\right )
\]
has the $i$-th coordinate 
\[
(c+1)+\lambda_i(a-nc-1)+\lambda_i^2(b+(n+1)c)
\]
for $i=1$ and $2$. Since $\lambda_i$ is a root of the characteristic polynomial 
\[
q_n(x)=x^3 - (n+1)x^2 + nx -1
\]
we have
\begin{align*} 
&  (c+1)+\lambda_i(a-nc-1)+\lambda_i^2(b+(n+1)c) \nonumber \\
&= \lambda_i(a+\lambda_i b) + [(n+1)\lambda^2_i-n\lambda_i+1]c + 1-\lambda_i  \\
& = \lambda_i (a+b\lambda_i+c\lambda^2_i) + (1-\lambda_i). 
\end{align*}
It follows that for element $(a,b,c) \in \Lambda_X$, we have
 \[
\pi_{xy}  \circ \Psi \left(  \begin{bmatrix}
a\\ b\\c
\end{bmatrix}\right ) \in Y \quad \mbox{and} \quad   \Psi  \begin{bmatrix}
a\\ b\\c
\end{bmatrix} \in \Lambda_Y.
\]
In addition, the map $\Psi: \Lambda_X \to \Lambda_Y$ is a bijection with the inverse
\[
\Psi^{-1} \left ( \begin{bmatrix}
a \\ b\\c 
\end{bmatrix} \right ) =\begin{bmatrix}
n & 1 & 0 \\
-(n+1) & 0 & 1 \\
1 & 0 & 0
\end{bmatrix}
\left(\begin{bmatrix}
a\\b\\c
\end{bmatrix}-\begin{bmatrix}
1 \\ -1 \\0
\end{bmatrix}\right).
\]
\begin{lemma}
The map $\Psi$ preserves the ordering of the lattice walk $\{\omega_0, \omega_1, \omega_2 \cdots\}$ corresponding to the orbits $\{p, T(p), T^2(p), \cdots\}$, i.e. 
\[
\pi_{z}(\omega_i)<\pi_{z}(\omega_j) \quad \mbox{if and only if} \quad \pi_{z} \circ \Psi(\omega_i)<\pi_{z} \circ \Psi (\omega_j).
\]
\end{lemma}
\begin{proof}
The proof follows directly from the calculation 
\[
\pi_{z} \circ \Psi \left( \begin{bmatrix}
a \\ b \\ c
\end{bmatrix} \right) = \lambda_3 (a+b \lambda_3 + c \lambda^2_3)+(1-\lambda_3) =  
\lambda_3 \thinspace \pi_{z} \left( 
\begin{bmatrix}
a\\b\\c
\end{bmatrix}
\right ) + (1-\lambda_3)
\]
where $\lambda_3>1$ is a root of the polynomial $q_n(x)$.\\
\end{proof}

Suppose $\omega_1 \in \Lambda(Y,L)$ and $q=\pi_{xy}(\omega_1)$. Consider the sequence $\{\omega_1, \omega_2 , \cdots\}$ of consecutive points of the lattice walk in $\Lambda_Y$. Let $\omega'_1=\Psi^{-1}(\omega_1)$ and $\{\omega'_1, \omega'_2, \cdots\}$ be the lattice walk in $\Lambda_X$ starting at $\omega_1'$.  We claim that
$$\omega_2 = \omega_1 + \Psi(\omega'_2-\omega'_1).$$ 
To see this, note that $\Psi$ is bijective and
$$
\Psi^{-1}(\omega_1+\Psi(\omega'_2-\omega'_1))=\omega'_1+\omega'_2-\omega'_1=\omega'_2 \in \Lambda_X.
$$
Also note that $\omega_2'$ is the point in $\Lambda(X,L)$ of smallest $z$-coordinate after $\omega_1'$. 


\old{
Let $p_1=\pi_{xy}(\omega_1)$ and $p_2 = \pi_{xy} \circ \Psi( \omega'_2)$. We show that $p_2$ is the image of the first return map $\hat T|_Y (p_1)$. Prove by contradiction. Suppose $\exists$ $\omega_j \in \Lambda_Y$ such that $\pi_{z}(\omega_1)<\pi_{z}(\omega_j)<\pi_{z} \circ \Psi(\omega'_2)$. Then $\pi_{z}(\omega_j -\omega_1)<\pi_{z}(\Psi(\omega'_2)-\omega_1)$. Since the map $\Psi$ preserves the ordering of the lattice walk, $\Psi^{-1}$ also preserves the order. Thus, 
\[
\pi_{z}(\Psi^{-1}(\omega_j)-\omega'_1)<\pi_{z}\circ \Psi^{-1}(\Psi(\omega'_2)-\omega_1)=\pi_{z} (\omega'_2-\omega'_1).
\]
By the construction $\pi_{z}(\omega'_2-\omega'_1)=\pi_{z}(\omega'_2)-\pi_{z}(\omega'_1)=\min\{\pi_{z}(\omega)-\pi_{z}(\omega'_1)| \thinspace \omega \in  \Z^3\}$. Hence we have $p_2=\hat T|_Y (p_1)$.
}
\qed

\section{Multi-stage REMs}\label{sec:multi}

\subsection{Construction}\label{sec:cons-multi}
Recall that for $n\ge 6$ there is a PV REM $T_{M_n}$ associated to a matrix
\[
M_n=\begin{bmatrix}
0 & 1 & 0 \\
0 & 0 & 1 \\
1 & -n & n+1
\end{bmatrix}.\]
Let \(\{v_i'\}_{i=0}^6\) be the translation vectors of \(T_{M_n}\) constructed as in section \ref{sec:cons}. Certain products of the matrices in \(\S\) define REMs with the same combinatorics as \(T_{M_n}\) (recall that the family of REMs defined by single matrices in 
\(\S \) all have the same combinatorics). 

Let \(W\in \mathcal M\) and define the normalized eigenvectors of \(W\) associated to $\lambda_1,\lambda_2$ to be
\[
\xi_1=(1, x, x') \quad \mbox{and} \quad \xi_2=(1, y, y'),
\]
scaled so that the first coordinate is \(1\).  Lemma \ref{lem:monoid} establishes that $W$ has real and positive eigenvalues. Since \(W\) is an integer matrix the eigenvectors are also real and we can define the projection \(\pi_{xy}: \Z^3 \to \R^2\) by
\[
\pi_{xy}: \mathbf x \mapsto (\mathbf x \cdot \xi_1,\  \mathbf x \cdot \xi_2).
\]

There is a dynamical system induced by \(W\) whose translation vectors are 
\[V=\{ v_i=\pi_{xy}(\eta_i), \text{ for } i=0,1,\ldots 6\}\]
where \(\mathcal E=\{\eta_i\}_{i=0}^6\) are 
\begin{align*}
\begin{aligned}
\eta_0 & =(-1,1,0), \quad \eta_1=(0,1,0), \quad \eta_2=\eta_0+\eta_1=(-1,2,0) \\ 
\eta_3 & =(1,-3,1), \quad \eta_4=\eta_0+\eta_3=(0,-2,1), \\
\eta_5 & =\eta_1+\eta_3=(1,-2,1), \quad \mbox{and} \quad \eta_6  =\eta_0+\eta_1+\eta_3=(0,-1,1)
\end{aligned}
\end{align*}
(their representations in \(\Z^3\) are the same as in (\ref{lattice-walk-vectors})). 

\begin{defn}\label{defn:admissible}
We say that \(W\) is an \textbf{\emph{admissible}} matrix when \(\xi_1,\xi_2\in \R^3_{>0}\) and the following two conditions are satisfied for each \(i=0,1 \ldots, 6\):
\begin{enumerate}
\item \(v_i\in (-1,1)^2\) 
\item \(v_i\) and \(v_i'\) lie in the same quadrant of \(\R^2\).
\end{enumerate}

We let $T_W$ be the REM constructed with these translation vectors whose partition is constructed using the method in section \ref{sec:cons};
we call it an {\bf admissible REM}. Let \(\mathcal M_A\subset \mathcal M\) be the subset of admissible matrices.
\end{defn}
The tiles in the partition \(\mathcal A = \{ A_0,\ldots, A_6\}\) associated to \(T_W\) are  
\begin{enumerate}
\setcounter{enumi}{-1}
\item \(A_0=[1- x,1] \times [1-y,1]\)
\item  \(A_1= [0, 1- x]\times [0, 1- y] \)
\item \(A_2= \left([1-2x,1-x] \times[1-y,2-2y] \right) \cup \left( [1-x,1]\times[0,1-y] \right)  \)
\item \(A_3 =[0,3x-x'] \times [-1+3 y-y',1] \)
\item \(A_4 = [3x-x', 1-x] \times [2y-y',1]\)
\item \(A_5=[0, 2x-x'] \times[ 1- y, -1+3y-y']\)
\item \(A_6 =\big ( [1-2x,3x-x'] \times[2-2y,-1+3y-y'] \big)  \\ 
\cup \big ( [2x - x', 1-2 x]\times[1-y,-1+3 y-y'] \big ) \\
\cup\big( [3x-x',1-x]\times[2-2y,2y-y'] \big) \).
\end{enumerate}

Within \(\mathcal M_A\) there is a subset  \(\mathcal M_R\) of matrices whose resulting REMs are renormalizable. Suppose \(W\in \mathcal M_A\) written in terms of generators as \(W= M_{n_L}M_{n_{L-1}} \cdots M_{n_1}\) with each \(M_{n_i}\in \S\). We develop an \(L\)-step renormalization scheme for the multistage REM \(T_W\).

To simplify the exposition, we introduce a notation for partial matrix products. Let $W_1 = M_{n_1}$ and set
\[
W_k = M_{n_{k}} \cdots M_{n_{1}}, \quad  \text{ for } k = 1,2, \ldots L
\]
with \(W = W_L\). For $k=1, 2, \cdots, L$, define the vectors $\xi^k_1=(1, x_k, x'_k)$ and $\xi^k_2=(1, y_k, y'_k)$ to be scalings of 
\[
W_k  \xi_1 \quad \mbox{and} \quad W_k  \xi_2
\]
normalized so that the first coordinate is \(1\). Define the projection $\pi_{xy}^k: \Z^3 \to \R^2$ by the formula
\[
\pi_{xy}^k: \mathbf x \mapsto (\mathbf x \cdot \xi_1^k,\  \mathbf x \cdot \xi_2^k).
\] 
At the \(k\)-th stage the translation vectors
\[V_k=\{ v_i^k=\pi_{xy}^k(\eta_i), \text{ for } i=0,1,\ldots 6\}\]
define a REM \(T_{W_k}\) with partition \(\mathcal A_k= \{ A_0, \ldots A_6\}\) where \(x=x_k\),\(x'=x_k'\), \(y= y_k\) and \(y'=y_k'\).

\begin{defn}\label{defn:multi}
An admissible REM \(T_W\) is a \textbf{\emph{multi-stage REM}} when the two conditions:
\begin{enumerate}
\item \(v_i^k\in (-1,1)^2\) 
\item \(v_i^k\) and \(v_i'\) lie in the same quadrant of \(\R^2\) 
\end{enumerate}
are satisfied for all \(i=0,1 \ldots, 6\) and all \(k=1,2 \ldots, L\). 
\end{defn}
At every stage \(i\) the REM \(T_{W_i}\) has the same combinatorics as \(T_W\). We prove that a multistage REM associated to a word \(W\) decomposed into a product of \(L\) generating elements \(\S\) has a \(L\)-step renormalization scheme.

\begin{thmu}[Detailed statement of Theorem \ref{thm:mult}]
Let $W=M_{n_L} M_{n_{L-1}}\cdots M_{n_{1}}\in \mathcal M_R$ and \(T_{W_k}:X\to X\) be the \(k\)-th stage of the multistage REM \(T_W\). For each stage $k$ let $Y_k=A^k_0$ be the rectangle of width $x_k$ and height $y_k$ whose upper left vertex is (1,1). Then 
\[
\widehat T_{W_k}|_{Y_k}=\phi^{-1}_k \circ T_{W_{k+1}} \circ \phi_k
\]
where $\phi_k: Y_k \to X$ is defined by
\[
\phi_k: (x,y) \mapsto \left(\frac{x+x_k-1}{x_k}, \frac{y+y_k-1}{y_k}\right).
\] 
\end{thmu}
Figure \ref{fig:seq} shows the sequence of partitions in the renormalization scheme for a multistage REM with four stages.

\begin{figure}[h]
  \begin{center}
    \includegraphics[width=\textwidth]{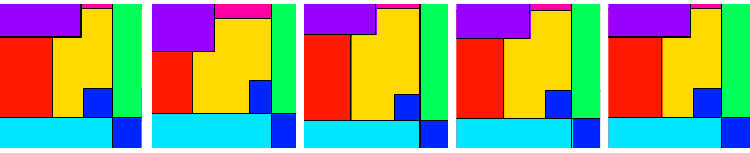}
\caption{The multi-stage REM \(T_W\) and associated REMs \(T_{W_1},T_{W_2},T_{W_3}\) and \(T_{W_4}=T_W\) with \(W=M_7 M_7 M_8 M_6\).}
\label{fig:seq}
  \end{center}
\end{figure}

\old{
\begin{figure}[h]
  \begin{center}
    \includegraphics[scale=0.6]{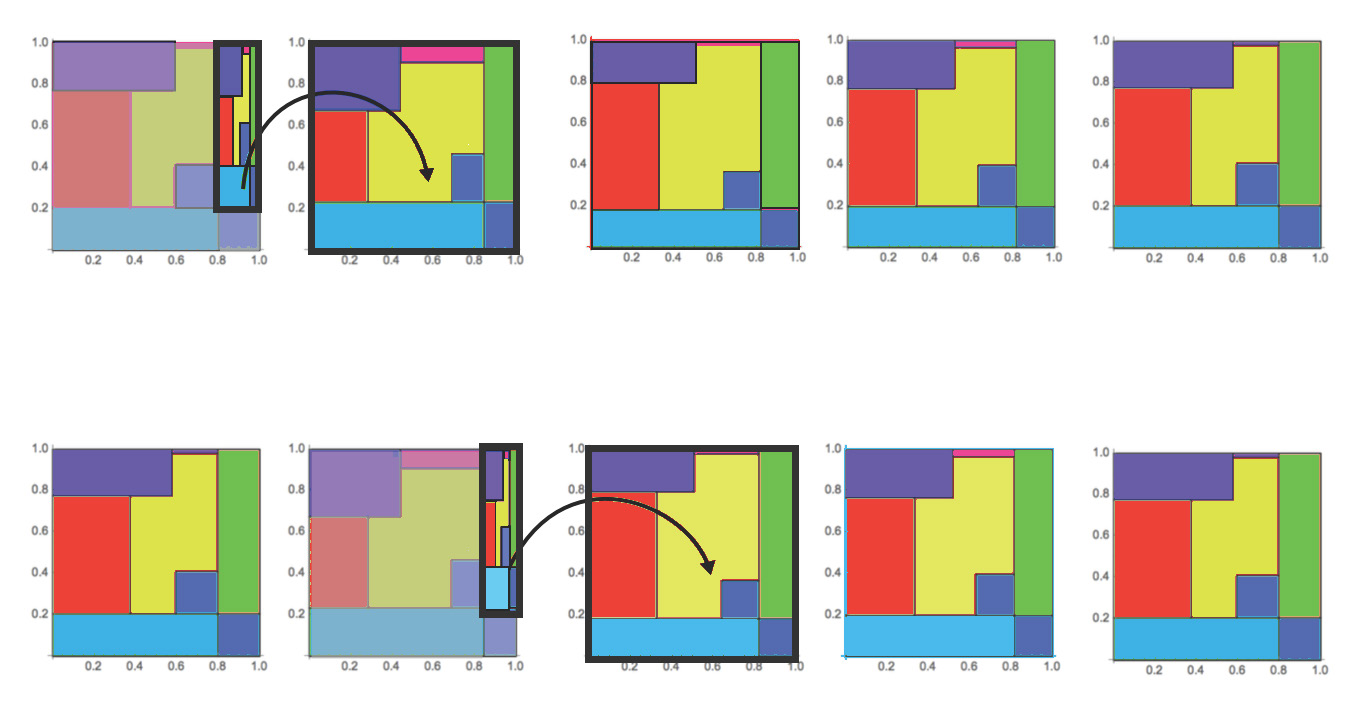}
\caption{First 2 steps of the renormalization scheme. In the first row the first return set $Y_0$ is shown bordered in black and in the second row the first return set $Y_1$ is shown bordered in black. \td{What are the other panels of this Figure? Is it related to the previous figure?}}
\label{fig:seq1return}
  \end{center}
\end{figure}}

\begin{figure}[h]
  \begin{center}
    \includegraphics[scale=0.6]{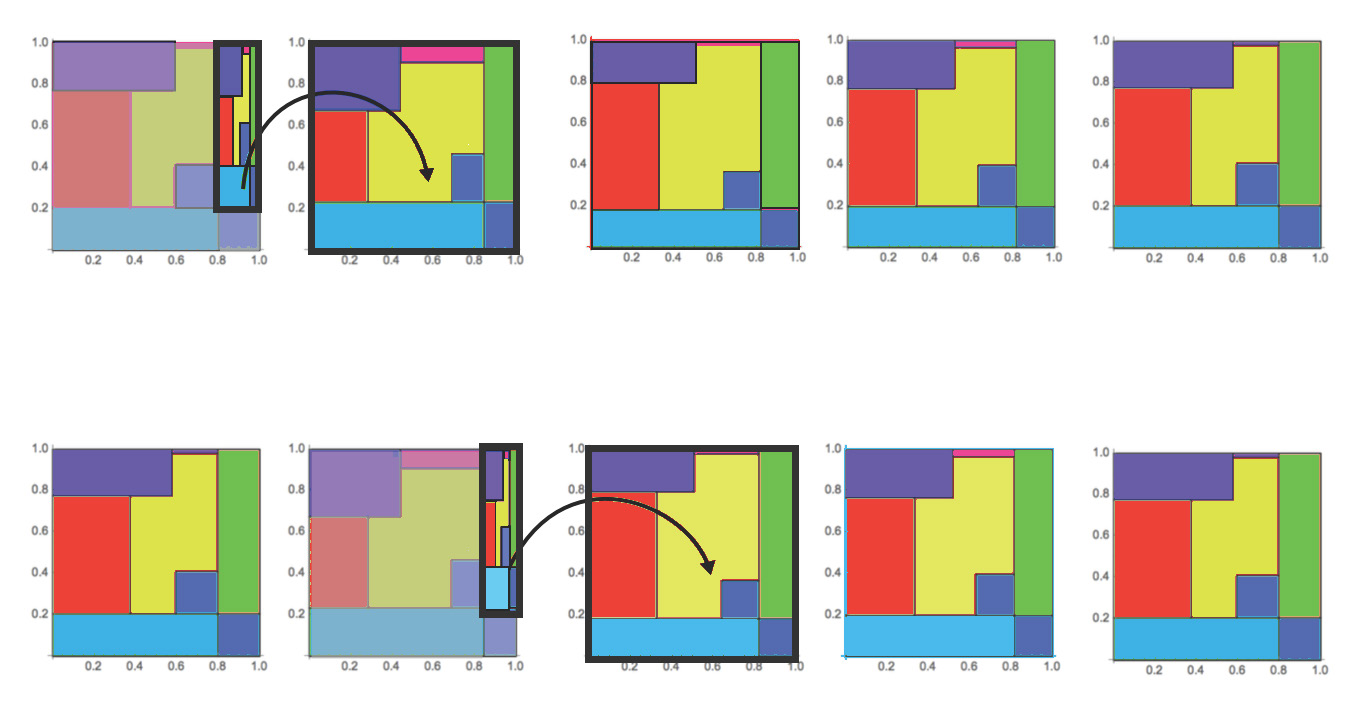}
\caption{Detailed view of the renormalization scheme shown in Figure \ref{fig:seq}. The first row shows the first return set \(Y_0\) bordered in black with the partition induced by the first return map overlayed. An arrow points to the REM in the sequence to which the first return map is affinely conjugate. The second row shows the same for \(Y_1\).}
\label{fig:seq1return}
  \end{center}
\end{figure}

\begin{proof}[Proof of theorem \ref{thm:mult-mi}]
Let \(W\in \mathcal M_A\) with eigenvalues \(\lambda_1,\lambda_2,\lambda_3\) and associated eigenvectors \(\xi_1,\xi_2,\) and \(\xi_3\) normalized so that the first coordinate is one. The multistage REM \(T_W\) can be constructed using cut-and-project sets with
\[\Lambda(X, L)=\{\mathbf x \in \Z^3: \ \pi_{xy}(\mathbf x) \in X\}\]
where the projection \(\pi_{xy}\) is defined as above. Therefore the same method as used in the proof of Theorem \ref{minimality} can be used to show that multistage REMs are minimal. However it remains to show that \(\pi_{xy}(\Lambda(X,L))\) is dense in \(X\). 
This follows from irreducibility: by admissibility, $\pm1$ are not eigenvalues of $W$, so the characteristic polynomial of $W$ is irreducible over $\Q$.
This implies that $W$ cannot have a proper $\Q$-invariant subspace, and thus the projection $\pi_{xy}(\Lambda(X,L))$ is dense.
\end{proof}

\subsection{\(\mathcal M\) is a monoid of Pisot matrices}\label{sec:monoid}

We prove Lemma \ref{lem:monoid} establishing that \(\mathcal M\) is a monoid of Pisot matrices. 

\begin{proof}[Proof of lemma \ref{lem:monoid}]
For a  \(3\times 3 \) matrix \(M\) label its eigenvalues \(\lambda_1(M), \lambda_2(M)\), and \(\lambda_3(M)\) and assume that they are ordered by increasing modulus. Let \(W=M_{n_0} \cdots M_{n_{L-1}} \) where each \(M_{n_i}\in \S\).

By a change of basis we have
\[
P_n=S^{-1}M_nS=\begin{bmatrix}
0 & 1 & 0 \\
0 & 1 & 1 \\
1 & 0 & n
\end{bmatrix} \quad 
where \quad  S = 
\begin{bmatrix}
1 & 0 & 0 \\
0 & 1 & 0 \\
0 & 1 & 1
\end{bmatrix}.
\]
The matrix \(P_n\) is primitive (has a strictly positive power) because 
\[ P_n^3 = \begin{bmatrix} 1 & 1 & 1+n \\ 1+n & 2 & 1+n+n^2 \\ n^2 & 1+n & 1+n^3\end{bmatrix}\]
therefore by the Perron-Frobenius theorem \(\lambda_3(P_n)>1\). It follows that that the leading eigenvalue of the product $P=P_{n_0} \cdots P_{n_{L}}\cdots P_{n_2}P_{n_1}$ is real and larger than $1$ since it is a finite product of primitive matrices and therefore primitive.  Note that the products $P=P_{n_0} \cdots P_{n_{L-1}}$ and \(W=M_{n_0} \cdots M_{n_{L-1}}\) have the same eigenvalues. Thus, we conclude that the leading eigenvalue $\lambda_3(W)$ is real and larger than $1$.

Arguing similarly as in the previous paragraph, we can use the Perron-Frobenius theorem to show \(\lambda_1(M_n)>0\):
by a change of basis of \(M_n^{-1}\) we have
\[
Q_n=A^{-1}M_n^{-1} A=\begin{bmatrix}
0 & 1 & 0 \\
0 & 2 & 1 \\
1 & -5+n & -2+n
\end{bmatrix} \quad 
where \quad  A = 
\begin{bmatrix}
0 & 2 & 1 \\
0 & 1 & 0 \\
1 & 0 & 0
\end{bmatrix}.
\]
Note that \(Q_n\) is primitive because
\begin{align*}
Q_n^3 & = \begin{bmatrix} 1 & -1+n & n \\ n & -1+(-3+n)n & -1 + (-1+n)n \\ -1 +(-3+n)n  &5 + (-5+n)(-1+n)n & 3+ (-4+n)n^2 \end{bmatrix}
\end{align*}
which is positive for \(n\geq 6\). By the Perron-Frobenius this implies \(1/\lambda_1(Q_n)>1\) and thus \(\lambda_1(Q_n)\) is real, positive, and less than \(1\). Using the same argument as above, the product $Q=Q_{n_1}Q_{n_2}\cdots Q_{n_L} = A^{-1}(M_{n_L} \cdots M_{n_2}M_{n_1})^{-1} A$ is primitive and therefore its leading eigenvalue is real and larger than one. Thus we find \(0<\lambda_1(W)<1\).

It remains to show \(\lambda_2(W)<1\). For simplicity we show this for the conjugated matrices \(P_n\).
 The characteristic polynomial $q_P$ of the matrix $P$ has the form
\begin{eqnarray*}
q_P(x)&=&x^3 - \Tr(P) x^2 + b(P) x -1\\
&=& x^3-(P_{1,1}+P_{2,2}+P_{3,3}) x^2 +([P]_{1,1}+[P]_{2,2}+[P]_{3,3})x-1
\end{eqnarray*}
where \(P_{i,j}\) denotes the entry of the matrix in the \(i\)-th column and \(j\)-th row and \([P]_{i,j}\) denotes the minor of $P$ obtained by deleting the \(i\)-th row and \(j\)-th column (i.e., the determinant of the submatrix obtained by deleting row \(i\) and column \(j\)). Evaluating \(q_P\) and its derivatives at \(-1\) and \(1\) we find 
\[
q_P(-1)=-1, \quad q_P'(0)=b(P), \quad q_n(1)=-\Tr(P)+b(P) \quad \text{and} \quad q_n'(1)=3-2\Tr(P)+b(P) . 
\]
Since \(\lambda_1>0\) we find that \(\lambda_2<1\) as long as \(b(P)<\Tr(P)\).

In order to prove that \(b(P)<\Tr(P)\) we need one fact about the signs of the minors of \(P\). We claim that \(P^{-1}\) can be written as
\[P^{-1}=
\begin{bmatrix}
a_{11} & -a_{12} & a_{13} \\
a_{21} & -a_{22} & a_{33} \\ 
-a_{31} & a_{32} & -a_{33}
\end{bmatrix}  
\]
where $a_{ij}$ are non-negative integers for $i,j=1, 2$ and $3$. The proof of this fact is postponed until after our main argument in which we prove $b(P)<\Tr(P)$. For an arbitrary \(3\times3\) matrix \(A\), the inverse can be calculated in terms of the minors of \(A\)
\begin{align*}A^{-1}&=
\begin{bmatrix}
a & b & c \\ d & e & f \\ g & h & i 
\end{bmatrix}  ^{-1}
= \frac{1}{\text{det} (A)}  \begin{bmatrix} [A]_{1,1}  &-  [A]_{1,2} & [A]_{1,3} \\ -[A]_{2,1}  &  [A]_{2,2} & -[A]_{2,3} \\ [A]_{3,1}  &-  [A]_{3,2} & [A]_{3,3}  \end{bmatrix}.
\end{align*}

Since \([P]_{2,2}\leq 0\) and \([P]_{3,3}\leq 0\), we have
\begin{align*}
b(P) & = [P]_{1,1}+[P]_{2,2}+[P]_{3,3} \leq [P]_{1,1}
\end{align*}
Thus \([P]_{1,1} \leq \Tr(P) \) implies that \(b(P) \leq \Tr(P)\). We use induction on the length of the product  \(P\) to prove that \([P]_{1,1} \leq P_{3,3}\). Since \(P\) has non-negative entries this will imply \([P]_{1,1} \leq \Tr(P)\).

In the base case, \(P=P_{n_0}\), and we have 
\begin{align*}
[P_{n_0}]_{1,1} =n_0 \leq n_0+1= P_{3,3}. 
\end{align*}
For the inductive step assume that \([P]_{1,1}<\Tr(P)\) for any \(P\) a product of \(L-1\) matrices. Let \(P'\) be a product of \(L\) matrices. We can write \(P'=PP_{n_L}\) where 
\[P= P_{n_0} \cdots P_{n_{L-1}} =\begin{bmatrix} 
x_{11} & x_{12} & x_{13}\\
x_{21} & x_{22} & x_{23}\\
x_{31} & x_{32} & x_{33}
\end{bmatrix}.\]
The matrix \(P'\) has the form
\begin{align*}
P'=P P_{n_L}=  \begin{bmatrix}
x_{13}  & x_{11}+x_{12} & x_{12}+n_L x_{13}\\
x_{23} & x_{21}+x_{22} & x_{22}+n_L x_{23}\\
x_{33}  & x_{31}+x_{32} & x_{32}+n_L x_{33}
\end{bmatrix}.
\end{align*}
Now we have
\begin{align*}
 [P']_{1,1}& =x_{21}x_{32}-x_{22}x_{31}+n_L(x_{22}x_{33}-x_{23}x_{32})+n_L(x_{21}x_{33}-x_{23}x_{31}) \\
& =[P]_{3,1} -[P]_{2,1}n_L + [P]_{1,1} n_L \\
& \leq [P]_{1,1}n_L \\
& \leq x_{33} n_L \\
& \leq x_{32} + x_{33}n_L =P'_{3,3}.
\end{align*}
Between lines three and four we applied the inductive hypothesis and between lines four and five we used the fact that the matrix has non-negative entries.

Next we prove the fact about the signs of the entries of \(P^{-1}\). Label the entries of \(P^{-1}\) as 
\[P^{-1} =  \begin{bmatrix}
a_{11} & -a_{12} & a_{13} \\
a_{21} & -a_{22} & a_{33} \\ 
-a_{31} & a_{32} & -a_{33}
\end{bmatrix}   \]
where \(a_{ij} \geq 0 \).
First we use induction on the length of the matrix product to show the following six inequalities 
\begin{align*}
a_{1j} &> 3 a_{2j} \text{ for } j=1,2, \text{ or } 3 \\
a_{1j} &> 3 a_{3j} \text{ for } j=1,2, \text{ or } 3.
\end{align*}
In the base case we have 
\[ P_{n_1}^{-1} = \begin{bmatrix} n_1 & -n_1 & 1 \\ 1 & 0 & 0 \\ -1 & 1 & 0   \end{bmatrix}.\]
Since \(n \geq 6 \) the inequalities hold by inspection. For the inductive step let \(P' = P P_{n_L}  \) be a product of \(L+1\) matrices. Then we have 
\begin{align*}
P'^{-1} & = P_{n_L}^{-1}P^{-1} = \begin{bmatrix} - a_{31} + n_L (a_{11}-a_{21})  & a_{32} + n_L (a_{22}-a_{12})  & - a_{33} + n_L (a_{13}-a_{23})  \\  a_{11} &-a_{12} & a_{13} \\  a_{21}-a_{11} & a_{12}-a_{22} & a_{23}-a_{13} \end{bmatrix}.
\end{align*}
Using the inductive hypothesis we have
\begin{align*}
(a_{11}-a_{21} ) n_L - a_{31}>a_{11}(n_L-\frac{1}{3}n_L-\frac{1}{3}) > 3 a_{11}
\end{align*}
since \(n_L \geq 6 \). This shows \(a_{11} > 3 a_{21}\). For \(P'\), again using the inductive hypothesis 
\begin{align*}
(a_{11}-a_{21} ) n_L - a_{31}&> a_{11}(n_L - \frac{1}{3}) - a_{21} n_L > (n_L-1)(a_{11}-a_{21}) -a_{21}+ \frac{2}{3} a_{11} \\
& > (n_L-1)(a_{11}-a_{21})
\end{align*}
and \(n_L \geq 6 \) from which we deduce that \(a_{11}> 3 a_{31}\). The calculations in the proofs of the remaining four inequalities are identical. 

Finally we complete the proof of the signs of the entries of \(P^{-1}\). Once again we induct on the length of the matrix product. The base case holds by inspection. In the inductive step we compute the signs of the entries of the first column of \(P'^{-1}\). We have
\begin{align*}
- a_{31} + n_L (a_{11}-a_{21})> a_{11} (n_L - 1/3 -1/3) > 0
\end{align*}
and 
\begin{align*}
a_{21} - a_{11}<a_{11}(1-1/3)<0.
\end{align*}
Similar calculations show that the signs of the other entries are as stated.
\end{proof}

\begin{figure}[h]
\centering
\includegraphics[scale=0.3]{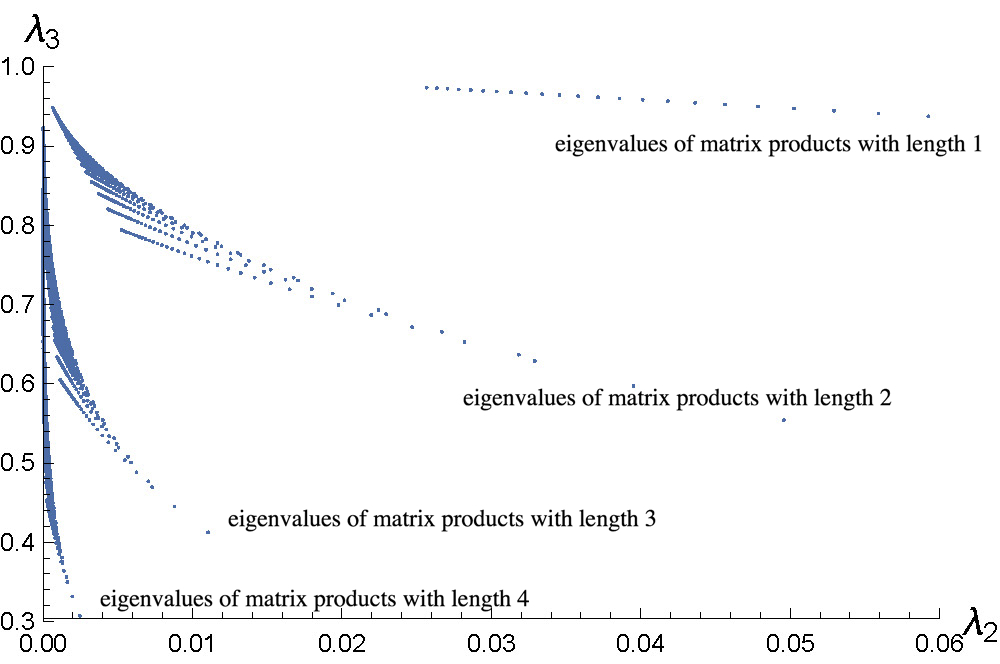}
\caption{Two of the three eigenvalues for matrices in the monoid. Each cluster of points corresponds to matrix products with the same length.} \label{eigenvalue-plot}
\end{figure}

\subsection{Proof of Theorem \ref{thm:mult}}
Let $W=M_{n_{L}}\cdots M_{n_1}$ be a matrix in $\mathcal M_R$ (Section \ref{sec:cons-multi}) and $\lambda_1$, $\lambda_2$, $\lambda_3$ be the eigenvalues of $W$ such that $0<\lambda_1<\lambda_2<1<\lambda_3$ (Lemma \ref{lem:monoid}). Let $\xi_1=(1, x_0, x_0')$ and $\xi_2=(1,y_0, y_0')$ be eigenvectors of $W$ with respective eigenvalues $\lambda_1$ and $\lambda_2$. Define the product $W_{k}=M_{n_{k}}\cdots M_{n_{1}}$ and $\xi_1^k=(1, x_k, x'_k)$ as a scaling of $W_k \xi_1$.

 Although the \(\xi_i\) are not eigenvectors they do satisfy the important property
\begin{align*}
M_{n_{k+1}} \xi_1^k = x_k \xi_1^{k+1}
\end{align*}
because
\begin{align*}
\begin{bmatrix}0 & 1 & 0 \\ 0 & 0 & 1 \\ 1 & -n_{k+1} & n_{k+1}+1 \end{bmatrix} \begin{bmatrix}1 \\ x_k \\ x_k'\end{bmatrix} & = \begin{bmatrix}x_k \\ x_k' \\  1-x_k n_{k+1} + x_k'(n_{k+1}+1) \end{bmatrix} \\
& = x_k \begin{bmatrix}1 \\ x_k' /x_k\\  \left(1-x_k n_{k+1} + x_k'(n_{k+1}+1) \right)/x_k \end{bmatrix} \\
& = x_k \xi_1^{k+1}.
\end{align*}

Similarly we define $\xi_2^k=(1, y_k, y'_k)$ be a scaling of $W_k \xi_2$. Recall the projection $\pi_{xy}^k$ at stage $k$ where $1\leq k \leq L$ is defined by the formula
\[
\pi_{xy}^k(\mathbf x) = (\xi^k_1 \cdot \mathbf x, \ \xi^k_2 \cdot \mathbf x)
\]
Let $Y_k $ be the set $A^k_0$ of the multistage REM $T_W$ associated to $W$. More precisely, $Y_k$ is a rectangle of width $x_k$ and height $y_k$ and the upper right vertex of $Y_k$ is $(1,1)$.  Define 
\[
\Lambda_{X_k}=\{\mathbf x \in \Z^3 | \ \pi^k_c(\mathbf x) \in X \} \quad \mbox{and} \quad \Lambda_{Y_k}=\{\mathbf x \in \Z^3 | \ \pi^k_c(\mathbf x) \in Y_k \}.
\]

Define the affine map 
\[
\Psi_k: 
\begin{bmatrix}
a \\ 
b \\
c
\end{bmatrix}  \mapsto 
(M_{n_{k+1}})^T 
\begin{bmatrix}
a \\ 
b \\ 
c
\end{bmatrix}+ 
\begin{bmatrix}
1 \\
-1 \\ 
0
\end{bmatrix}.
\]
We claim that $\Psi_k: \Lambda_{X_{k+1}} \to \Lambda_{Y_k}$ is a bijection. To prove the statement, we first show that $\Psi_k(\mathbf x) \in \Lambda_{Y_k}$ for $\mathbf x \in \Lambda_{X_{k+1}}$, i.e. 
\[
\pi_{xy}^k\circ \Psi_k(\omega) = (\xi^k_1 \cdot \Psi_k(\omega), \ \xi^k_2 \cdot \Psi_k(\omega)) \in (1-x_k,1) \times (1-y_k, 1).
\]
We compute the $x$-component of the projection $\pi_{xy}^k \circ \Psi_k(\omega)$ 
\begin{eqnarray*}
\xi^k_1 \cdot \Psi_k(\omega) &=& \xi_1^k \cdot \left (M^T_{n_{k+1}} \omega + \begin{bmatrix}1 \\ -1 \\ 0\end{bmatrix} \right )\\
&=& M_{n_{k+1}}\xi^k_1 \cdot \omega + \xi_1^k  \cdot (1,-1,0)\\
&=& x_k \ \xi_1^{k+1} \cdot \omega + 1-x_k
\end{eqnarray*}
By the assumption $\omega \in \Lambda_{X_{k+1}}$, we have $\xi_1^{k+1} \cdot \omega \in (0,1)$. Therefore we conclude that 
\[
\xi_1^k \cdot \Psi_k(\omega) \in (1-x_k, 1).
\]
Using the same argument, we can show that the $y$-component of $\pi_{xy} \circ \Psi_k (\omega) \in (1-y_1, 1)$. Moreover, the inverse $\Psi^{-1}$ is given by 
\[
\Psi_k^{-1}: \omega \mapsto (M^T_{n_{k+1}})^{-1} \left (\omega- \begin{bmatrix} 1 \\ -1 \\ 0 \end{bmatrix} \right ).
\] 
Thus, the map $\Psi_k: \Lambda_{X_{k+1}} \to \Lambda_{Y_k}$ is a bijection.

We apply the same argument as in Section \ref{prf:single} to show the renormalization of multistage REMs. Here we show that $\Psi_k$ corresponds to a return map of the multistage REM $T_{W_k}$. Let $\omega_0 \in \Lambda_{Y_k}$ and $q_0=\pi_{xy}^k(\omega_0) \in Y_k$. Define $\omega_0'=\Psi_k^{-1}(\omega_0) \in \Lambda_{X_{k+1}}$ and $q_0'=\pi_{xy}^{k+1}(\omega'_0)$. Let $\{\omega'_0, \omega'_1, \cdots\}$ be a sequence of consequence points of the lattice walk in $\Lambda_{X_{k+1}}$ where $\omega'_1\in \Lambda_{X_{k+1}}$ and 
\[
\pi^{k+1}_c(\omega'_j)=T^j_{W_{k+1}}(q).
\]
We have 
\[
q_1=q_0+\pi^k_c \circ \Psi_k (\omega'_1-\omega'_0) \in Y_k
\]
since 
\begin{eqnarray*}
q_1=q_0+\pi^k_c \circ \Psi_k (\omega'_1-\omega'_0) &=& \pi_{xy}^k(\omega_0)+ \pi_{xy}^k \circ \Psi_k(\omega'_1-\omega'_0)\\
&=& \pi_{xy}^k(\omega_0+\Psi_k(\omega'_1) -\Psi_k(\omega'_0)) \\
&=& \pi_{xy}^k (\omega_0-\omega_0+\Psi_k(\omega'_1))\\
&=&\pi_{xy}^k \circ \Psi_k(\omega'_1) \in Y_k.
\end{eqnarray*} 

Moreover, because the map $\Psi_k$ is bijective,  the point $q_1$ must be the image of the first return map $\hat T_{W_k}(q_0)|_{Y_k}=q_1$. It means that 
\[
\hat T_{W_k}|_{Y_k} = \phi_k^{-1} \circ T_{W_{k+1}} \circ \phi_k
\]
where the affine map $\phi_k$ maps $Y_k$ to the unit square $X=X_k$.

\section{Parameter space of multistage REMs}
The space of multistage REMs is a subset of \(\R^4\). It can be naturally parametrized by the two eigenvectors associated to a matrix in \(\mathcal M_R\) whose associated eigenvalues are less than one. Let \(\lambda_1,\lambda_2\) and \(\lambda_3\) denote the eigenvalues of a matrix in \(\mathcal M_R\) ordered by increasing magnitude. Scale the eigenvectors of \(\mathcal M_R\) so that the first coordinate is \(1\). Let \((1,x,x')\) denote the eigenvector associated to the eigenvalue \(\lambda_1\) and let \((1,y,y')\) denote the eigenvector associated to the eigenvalue \(\lambda_2\). In Figure \ref{fig:cantor} we plot points in the parameter space with $(x,x',y)$-coordinates colored by their $y'$-coordinate.

\old{\begin{figure}[h]
\centering
\includegraphics[width=\textwidth]{fig/param}
\old{\caption{The parameter space of multi-stage REMs in coordinates $(x,x',y,y')$ where $x,x',y$ and $y'$ are entries of the eigenvectors of the matrix products in $\mathcal M_R$. The projections of the points corresponding to multi-stage REMs onto xx'-plane (left) and yy'-plane (right). There is a Cantor set structure in each of the figures. }}
\caption{The 4-dimensional parameter space of multi-stage REMs. Each point in coordinates \((x,x',y,y')\) corresponds to a pair of eigenvectors \((1,x,x')\) and \((1,y,y')\) of a matrix associated to a multi-stage REM. We plot the projections of the parameter space onto the \(xx'\)- and \(yy'\)-planes.} \label{fig:cantor}
\end{figure}}

\begin{figure}[h]
\centering
\includegraphics[scale=0.6]{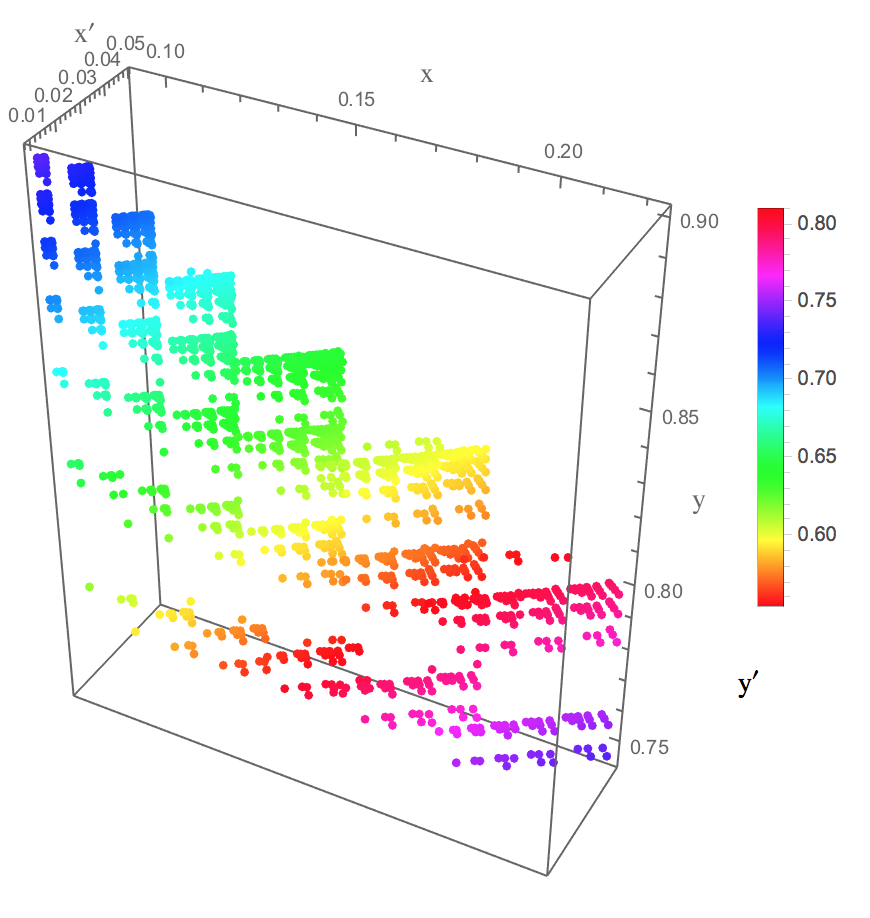}
\caption{The 4-dimensional parameter space of multi-stage REMs. Each point in coordinates \((x,x',y,y')\) corresponds to a pair of eigenvectors \((1,x,x')\) and \((1,y,y')\) of a matrix determining a multi-stage REM. Points are colored by the coordinate $y'$.} \label{fig:cantor}
\end{figure}

\begin{conj}
The closure of the parameter space of all renormalizable multistage REMs is a Cantor set in $\R^4$. 
\end{conj}

\section{Appendix}
We give a computational proof of Theorem \ref{thm:single} when $n=6$.
\begin{proof}
Note that $Y=\phi_n^{-1}(X)$ and we consider the first return map $\hat T_{M_6}|_Y$ restricting to each element $\hat A_k=\phi_n^{-1}(A_k)$ for $k=0,1, \cdots, 6$. Let $\rho: X \to X$ be the map given by $(x,y) \mapsto (\lambda_1  x,\  \lambda_2  y)$. Let $v_k$ be the translation vector on the set $A_k \in \mathcal A$. We show that the map $\hat T_{M_6}|_Y$ consists of translations by vectors $\rho(v_k)$ on each $\hat A_k^\circ$.

For each point in $\hat A_k$, we associate a symbolic sequence tracking its orbit until it returns to the set $Y$. More precisely, let $\Omega=\{0,1, \cdots, 6\}^{\Z^+}$ be the set of sequences in $\{0,1, \cdots, 6\}$, and define $\iota: X \to \Omega$ to be the coding
\[
\iota(p)= \alpha_0 \ \alpha_1 \cdots \ \alpha_{m}, \quad \mbox{for $\alpha_j\in \{0,1, \cdots, 6\}$ and $T^{m}(p) \in Y$},
\]
where $A_{\alpha_j}$ is the tile containing $T^j(p)$.
Define $\mathcal R_{\iota(p)}=\{q \in Y | \ \iota(q)=\iota(p)\}$ the maximal set of points with the same coding associated to $\iota(p)$. The first return map $\hat T|_Y$ restricting to $\mathcal R_{\iota(p)}$ is the 
translation given by
\[
p \mapsto p+ \pi_{xy} (\sum_{i=0}^{m-1} \eta_i).
\]

By computation, we obtain that $\hat A_0=\mathcal R_{05} \cup \mathcal R_{013} \cup \mathcal R_{031}$. The first return map restricting on $\hat A_0$ is the translation by the vector $v'_i=\pi_{xy}(\eta'_0)$ where 
\[
\eta'_0=(0,-1,1)=\eta_0+\eta_5=\eta_0+ \eta_1+\eta_3=\eta_0 + \eta_3+\eta_1.
\]
Then we have
\[
\pi_{xy} (\eta'_0)=(-\lambda_1+\lambda_1^2, -\lambda_2+\lambda_2^2)=(\lambda_1(-1+\lambda_1), \ \lambda_2(-1+\lambda_2)).
\]
Since $\eta_0=(-1,1,0)$ we have $v'_0=\pi_{xy}(\eta'_0)=\rho(\pi_{xy}(\eta_0))=\rho(v_0)$.

The element $\hat A_1=\mathcal R_{0131}$ so that the map $\hat T_{M_6}|_Y$ translates $\hat A_1$ by vector $\pi_{xy} (\eta'_1)$ where
\[
\eta'_1=\eta_0+ \eta_1+\eta_3+\eta_1=\eta_0+2\eta_1+\eta_3=(0,0,1).
\]
Therefore, 
\[
\pi_{xy}(\eta'_1)=(\lambda_1^2, \ \lambda_2^2) = \rho\circ \pi_{xy} (\eta_1).
\]

Since $\hat A_2=\mathcal R_{05231} \cup \mathcal R_{013231} \cup \mathcal R_{01325}$ and $\eta_5=\eta_1+\eta_3$, we have $\hat T_{M_6}|_Y: p \mapsto p+\pi_{xy}(\eta'_2)$ where
\[
\eta'_2=\eta_0+\eta_5+\eta_2+\eta_5=2\eta_0+3\eta_1+2\eta_3=(0,-1,2).
\]
It follows that 
\[
\pi_{xy} (\eta'_2) = (-\lambda_1+2\lambda_1^2, \ -\lambda_2+\lambda_2^2) = \left ( \lambda_1(-1+2\lambda_1), \lambda_2(-1+2\lambda_2) \right) =\rho\circ \pi_{xy} (\eta_2).
\]

The set $\hat A_3$ is the disjoint union of seven subsets 
\[
\hat A_3=\mathcal R_{0313265}\cup \mathcal R_{031665}\cup \mathcal R_{053265} \cup \mathcal R_{05665} \cup \mathcal R_{056235} \cup \mathcal R_{056613} \cup \mathcal R_{0562313}.
\]
Since 
\[
\eta_5=\eta_1+\eta_3 \quad \mbox{and} \quad \eta_6=\eta_0+\eta_1+\eta_3,
\]
the map $\hat T_{M_6}|_Y$ translates every well-defined point in $\hat A_3$ by the vector $\pi_{xy}(\eta'_3)$ for 
\[
\eta'_3=3\eta_0+4\eta_1+4\eta_3=(1,-5,4).
\]
Then we compute
\begin{eqnarray*}
\pi_{xy}(\eta'_3)&=&(1-5\lambda_1+4\lambda_1^2, \ 1-5\lambda_2+4\lambda_2^2)\\
&=& (\lambda_1^3-3\lambda^2_1+\lambda_1, \lambda_2^3-3\lambda_2^2+\lambda_2)\\
&=&  (\lambda_1 (1-3\lambda_1+\lambda_1^2), \ \lambda_2(1-3\lambda_2+\lambda_2^2)) \\
&=& \rho \circ  \pi_{xy} (\eta_3).
\end{eqnarray*}

The element $\hat A_4=\mathcal R_{03166613}$ with translation vector $\pi_{xy}(\eta'_4)$ under the first return map $\hat T_{M_6}|_Y$ where 
\[
\eta'_4=4\eta_0+5\eta_1+5\eta_3 =\eta'_0+\eta'_3=(1,-6,5).
\]
We have shown that for each $j=0,1$ and $3$, we have $\pi_{xy}(\eta'_j)=\rho\circ \pi_{xy}(\eta_j)$. Therefore, 
\begin{eqnarray*}
\pi_{xy}(\eta'_4)&=& \pi_{xy}(\eta'_0+\eta'_3)\\
&=&\pi_{xy}(\eta'_0)+\pi_{xy}(\eta'_3)\\
&=&\rho \circ \pi_{xy}(\eta_1)+ \rho \circ \pi_{xy}(\eta_3)\\
&=&\rho\circ \pi_{xy} (\eta_1+\eta_3) = \rho \circ \pi_{xy}(\eta_4).
\end{eqnarray*}

The set $\hat A_5$ is the union of seven disjoint subsets
\[
\mathcal R_{056613231}\cup \mathcal R_{0523613231} \cup \mathcal R_{052361325} \cup \mathcal R_{0523141325} \cup \mathcal R_{052316325} \cup \mathcal R_{0132316325} \cup \mathcal R_{013231665}.
\]
The vector 
\[
\eta'_5=4\eta_0+6\eta_1+5\eta_3=(1,-5,5).
\]
On the other hand, 
\[
\eta'_5=\eta'_1+\eta'_3.
\]
By the same argument as above, we have 
\[
\pi_{xy}(\eta'_5)= \rho \circ \pi_{xy}(\eta_5).
\]

The element $\hat A_6$ is partitioned into 19 subsets which are listed here
\begin{center}
$\mathcal R_{03166613231}, \ \mathcal R_{0566613231}, \ \mathcal R_{05623613231}, \ \mathcal R_{0562361325}, \ \mathcal R_{05623141325}, \ \mathcal R_{0566132325},$\\
$\mathcal R_{0 5 6 2 3 1 3 2 3 2 5 }, \mathcal R_{0 5 6 2 3 1 6 3 2 5}, \mathcal R_{0 5 2 3 6 1 3 2 3 2 5 }, \mathcal R_{0 5 2 3 2 3 1 3 2 3 2 5}, \mathcal R_{0 5 2 3 2 3 1 6 3 2 5}, \mathcal R_{0 1 3 2 3 1 6 6 6 5},$\\
$\mathcal R_{0 5 2 3 6 1 3 2 6 5}, \mathcal R_{0 5 2 3 2 3 1 3 2 6 5}, \mathcal R_{0 5 2 3 2 3 1 6 6 5}, \mathcal R_{0 5 2 3 1 4 1 3 2 6 5}, \mathcal R_{0 5 2 3 1 6 3 2 6 5}, \mathcal R_{0 1 3 2 3 1 6 3 2 6 5}, \mathcal R_{0 1 3 2 3 1 6 6 6 1 3}$.
\end{center}
Then 
\[
\eta'_6=5\eta_0+7\eta_1+6\eta_3=\eta'_0+\eta'_1+\eta'_3.
\]
The translation vector for the map $\hat T_{M_6}|_{Y}$ on $\hat A_6$ satisfies the equality 
\[
\pi_{xy}(\eta'_6)=\pi_{xy}(\eta'_0+\eta'_1+\eta'_3)=\rho\circ \pi_{xy}(\eta_0+\eta_1+\eta_3)=\rho\circ \pi_{xy}(\eta_6)
\]. 
\end{proof}

\begin{figure}[h]
\centering
\includegraphics[scale=0.5]{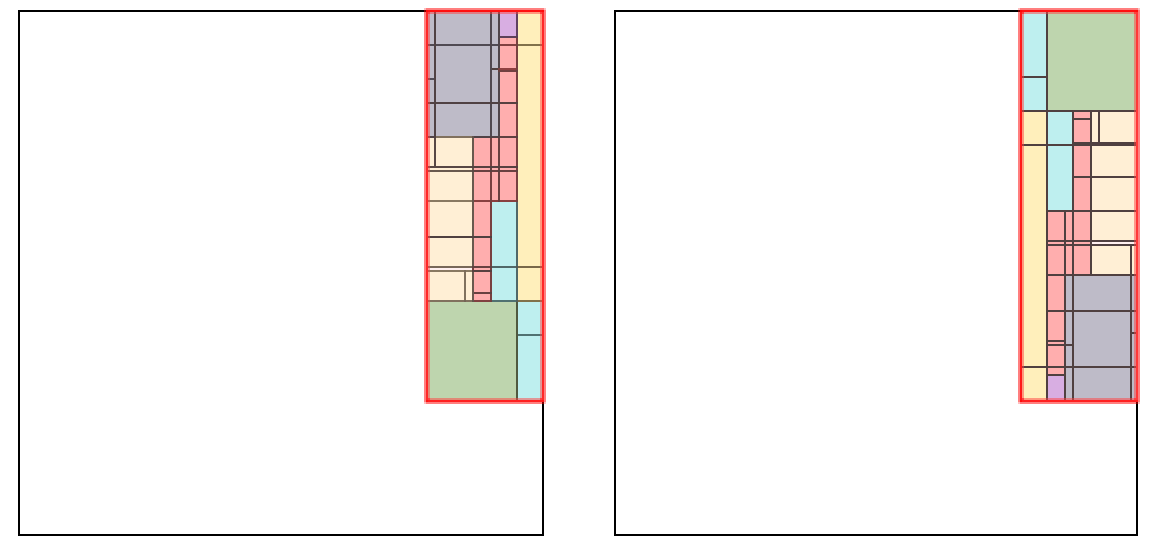}
\caption{The first return set $Y$ partitioned into tiles with the same symbolic codings}
\label{fig:return6}
\end{figure}

\subsection*{Acknowledgments}   
The authors would like to thank Richard Schwartz and Patrick Hooper for many helpful conversations. I. Alevy is supported by the
NSF grant  DMS-1713033. R. Kenyon is supported by the NSF grant DMS-1713033 and the Simons Foundation award 327929.

\bibliographystyle{hep}
\bibliography{mybib}

\end{document}